\documentclass[12pt]{article}
\usepackage{amsfonts} 
\usepackage{amssymb} 
\usepackage{amsmath}
\usepackage{amsthm} 
\usepackage{mathtools} 
\usepackage{graphicx} 
\usepackage{float}
\usepackage{bm}
\usepackage{color}
\usepackage{subcaption} 
\usepackage{cite} 
\usepackage{hyperref}
\hypersetup{colorlinks,linkcolor={red},citecolor={blue},urlcolor={red}} 
\usepackage{cleveref} 
\usepackage{enumitem}
\usepackage{multicol}
\usepackage{multirow} 
\usepackage{accents}

   \newcommand{\R}{\mathbb{R}}
\newcommand{\N}{\mathbb{N}} \newcommand{\Z}{\mathbb{Z}}
 \newcommand{\x}{\mathbf{x}}
\newcommand{\y}{\mathbf{y}} 
\newcommand{\bk}{\mathbf{k}} \newcommand{\A}{\mathcal{A}}
\newcommand{\B}{\mathcal{B}} \newcommand{\K}{\mathcal{K}}
 \newcommand{\D}{\mathcal{D}}
\newcommand{\bigO}{\mathcal{O}} \newcommand{\ci}{\mathrm{i}}

	\newcommand{\CC}{C\nolinebreak\hspace{-.05em}\raisebox{.4ex}{\tiny\bf
	+}\nolinebreak\hspace{-.10em}\raisebox{.4ex}{\tiny\bf +}}
	\def\CC{{C\nolinebreak[4]\hspace{-.05em}\raisebox{.4ex}{\tiny\bf ++}\
	  }}
	
	\newcommand{\ftilde}{\scalebox{1.35}{[}\widetilde{f}\; \scalebox{1.35}{]}}
	\newcommand{\fitilde}{\scalebox{1.35}{[}\widetilde{\phi}\, \scalebox{1.35}{]}}
	
	\newcommand{\pFfT}{\scalebox{1.35}{\em{(}}\! F,\widetilde{f}\, \scalebox{1.35}{\em{)}}}

\newcommand{\interior}[1]{%
	\accentset{\smash{\raisebox{-0.12ex}{%
				$\scriptstyle\circ$}}}{#1}%
	\rule{0pt}{2.3ex}%
}

\theoremstyle{plain} \newtheorem{theorem}{Theorem}
\newtheorem{lemma}{Lemma} 
\newtheorem{definition}{Definition} 
\newtheorem{remark}{Remark} \theoremstyle{remark}

\vspace{-2cm} \title{``Truncated Fourier Filtering'' method for fast and high-order evaluation of integrals and convolutions in general
domains} 
\author{Oscar P. Bruno\footnote{Computing and Mathematical
Sciences, Caltech, Pasadena, CA 91125, USA} \and
Ambuj Pandey\footnote{Indian Institute of Science Education and
Research Bhopal (IISER Bhopal)} 
\and
Krishna Y. Poojara$^*$ %
}
\date{}
\begin{document}
\maketitle
\begin{abstract} This paper introduces and analyzes a novel
  algorithm---Truncated Fourier Filtering (TFF)---for the fast, {\em high-order} accurate evaluation of standard integrals and
  convolutions involving piecewise-smooth (possibly discontinuous)
  integrands over general {\em $m$-dimensional domains} ($m \ge 1$)
  employing an $m$-dimensional {\em Cartesian grid}. For an $N$-point discretization, the method runs at a computational cost of
  $\mathcal{O}(N)$ operations for standard integrals and
  $\mathcal{O}(N \log N)$ operations for convolutions, following, in either case, a one-time $\mathcal{O}(N \log N)$ precomputation step (not required in dimension $m = 1$). The core idea underlying TFF is to approximate the characteristic function of the integration domain by
   a truncated Fourier expansion of it over a suitably extended periodic
  domain, and to evaluate the resulting integrals via trapezoidal
  quadrature on a Cartesian grid with an appropriately chosen
  discretization size.  Despite its conceptual simplicity, TFF attains
  high-order accuracy even for complex, possibly non-smooth or even
  non-Lipschitz geometries. A complete theoretical analysis is
  provided that establishes the superalgebraic convergence (i.e.,
  convergence faster than any negative power of $N$) of the overall
  approach.
\end{abstract}

\section{Introduction} \label{sec:Intro} This paper introduces an
efficient numerical integration method, \textit{Truncated Fourier
  Filtering (TFF)}, that enables superalgebraically fast
evaluation---that is, faster than any polynomial rate---of both
standard and convolution integrals involving possibly discontinuous
piecewise-smooth integrands over general domains in $m$-dimensional
space ($m \geq 1$), using an $m$-dimensional \emph{Cartesian
  discretization}.  Specifically, the method accurately computes
integrals of the form
\begin{equation}\label{Original_int_2D} \A[f] = \int_{\K} f(\x)\,d\x,
\end{equation} as well as convolution integrals
\begin{equation}\label{eq1} \B[\phi](\x) = \int_\K G(\x -
\y)\phi(\y)\,d\y,
\end{equation}
with smooth kernels $G$ and piecewise-smooth density functions $\phi$,
and over arbitrary compact sets $\K \subset \R^m$ with piecewise
smooth boundaries, which may include geometric singularities such as
corners and cusps. For an $N$-point discretization, the proposed
method evaluates standard integrals in $\mathcal{O}(N)$ computing
time and convolutions in $\mathcal{O}(N\log N)$ time, after a one-time
preprocessing stage of complexity $\mathcal{O}(N\log N)$ (not required in dimension $m=1$). The key idea underlying TFF, which is
rigorously analyzed in this paper, is to replace the characteristic
function of the integration domain by a truncated Fourier
representation of it on a suitable periodic domain containing $\K$, and then
evaluate the resulting integrals using trapezoidal quadrature on a
Cartesian grid of an appropriately chosen meshsize. Remarkably,
despite its conceptual simplicity and the slow convergence of the
Fourier expansion of a characteristic function, the resulting method
achieves high-order, and for sufficiently smooth data superalgebraic,
convergence even in the presence of complex geometries and
discontinuous integrands.

The evaluation of standard integrals~\eqref{Original_int_2D} and convolutions~\eqref{eq1} involving smooth or weakly singular kernels $G$ is of considerable practical and theoretical interest, as such problems arise in numerous applications. Prior
work on the former class of problems---particularly those involving
element-level integrals arising in finite element methods (FEM),
isogeometric analysis (IGA), discontinuous Galerkin (DG), and other
specialized finite element-based methods---is presented
in~\cite{fries2010extended, strouboulis2000design,
  mousavi2011numerical, lyness1977quadrature,
  joulaian2016numerical,sukumar2001modeling,glowinski1994fictitious,haslinger2009new,boroomand2023adaptive,bourhrara2025integrals}
and references therein. Significant contributions related to the
latter problem, on the other hand, have been made in connection with a
wide range of application areas, including in smooth-kernel
application settings such as image processing and diffusion
modeling~\cite{vogel1998fast, lu2010integral}, as well as weakly
singular kernel settings encountered in fluid dynamics, wave
propagation, and potential theory~\cite{liu2016collocation,
  hansen2002deconvolution, vogel1998fast, colton1998inverse,
  vico2016fast}, among others. We suggest that the simple yet accurate
and efficient equispaced-grid/trapezoidal-rule based approach
introduced in this paper could positively impact upon several of these
areas.

High-order quadrature schemes based on equispaced grids are highly
attractive due to their simplicity, ease of implementation, and
compatibility with fast Fourier transform (FFT)
techniques~\cite{qiang2010high}. The traditional $Q$-point composite
Newton-Cotes quadrature rules, which are also based on equispaced
meshes and are stable for $Q\leq 7$, exhibit a convergence rate of
order $Q+1$ when utilized over either one-dimensional intervals or
simple rectangular domains; the study of stable equispaced quadratures
over rectangular regions is taken up
in~\cite{huybrechs2009stable,platte2011impossibility,wilson1970necessary}.
Repeated application of one-dimensional Newton--Cotes or other
equispaced-mesh quadrature rules to evaluate integrals over
arbitrarily shaped $m$-dimensional domains typically yields poor
convergence, primarily owing to the mismatch between the underlying
Cartesian meshes and the generally curved boundaries of the
integration domains.

Among the quadrature rules based on equispaced meshes, the trapezoidal
rule demonstrates  exceptionally fast convergence when applied to smooth and periodic integrands
\cite{fornberg2021improving}---exhibiting exponential
convergence~\cite{trefethen2014exponentially} for integrands that are
both analytic and periodic. In view of its remarkable convergence
character, the trapezoidal rule has been widely utilized in various
applications, see \cite{trefethen2014exponentially,
bruno2001fast,anand2016efficient,duan2009high,wu2021corrected} and
reference therein. In higher dimensions, for smooth integrands $f$
that smoothly vanish outside the integration domain $\K$, the
quadrature problem can be reduced to integration over a periodic
hyper-rectangle $R\supset \K$---for which the high-order trapezoidal
rule accuracy results. Similarly, the compatibility of the trapezoidal
rule with the FFT also allows for fast and high-order evaluation of
convolution integrals~\eqref{eq1} with integral densities $\phi$ that
smoothly vanish outside the integration domain $\K$.

As is well known, however, the accuracy of the trapezoidal rule
degrades to second order in the absence of periodicity, and can drop
further to first order---even for smooth integrands---in the common
scenario, particularly prevalent in dimensions higher than one, where
the uniform discretization grid does not align precisely with the
domain boundary. Various techniques have been developed to overcome
this limitation and enhance the convergence rate for smooth,
non-periodic functions~\cite{rokhlin1990end, kapur1997high,
  alpert1999hybrid,
  slevinsky2015use,suzuki2025mobius,fornberg2023enhanced}. In
particular, the contribution~\cite{fornberg2023enhanced} extends the
high-order character of the trapezoidal rule for discontinuous
functions in one-dimensional space.  In the same spirit, the
boundary-corrected trapezoidal rule~\cite{rokhlin1990end,
  kapur1997high} achieves high-order convergence by incorporating
additional function values near the endpoints of the integration
interval. Related ideas have been used to handle singular convolution
kernels applied to smooth densities that vanish smoothly outside a
compact set~\cite{marin2014corrected, wu2021corrected,
  aguilar2002high, izzo2022corrected}. However, these methods fail to
deliver high-order convergence when the smooth density does not decay
smoothly to zero at the boundary of the integration domain.  To the
best of our knowledge, TFF is the only trapezoidal rule-based
algorithm that retains high-order accuracy when applied to
non-periodic integrands in dimensions $m>1$.


The TFF approach proceeds in a simple but somewhat surprising
manner. For example, in the 1D case ($m=1$), the TFF first
re-expresses the integral~\eqref{Original_int_2D} over the interval
$\K = [\alpha,\beta]$ as the integral $\int_R \chi^\K \widetilde f$
over a strictly larger interval $R = [a,b]$ which contains $\K$ in its
interior---$[\alpha,\beta]\subset (a,b)$,---where $\chi^\K$ denotes
the characteristic function of $\K$ and where $\widetilde f$ is an
$R$-periodic extension of $f$; see Remark~\ref{ext-window}. The
characteristic function is then approximated by a truncated Fourier
series $\chi^{\K}_F$, of order $F$ for certain integer orders $F$,
with periodicity domain $R$. The corresponding Fourier coefficients
must be computed to an accuracy commensurate with the prescribed
overall error tolerance. Typically, such accuracy cannot be
efficiently obtained by a direct application of the FFT algorithm, but
an efficient algorithm has been developed which can nevertheless
attain high order evaluation of such Fourier coefficients at a cost
of $\bigO(F\log F)$ operations. The product $\chi^{\K}_F f$ is
subsequently integrated numerically using the trapezoidal rule, for
instance with $N = 2F$ discretization points.  As
established in Theorem~\ref{theorem_main_hd}, for
$f \in C^\infty([a,b])$ the resulting algorithm exhibits
superalgebraic convergence---namely, for smooth functions $f$, the
error decays faster than any negative power of the discretization size
$N$---despite the notoriously slow convergence of the Fourier series
of $\chi^{\K}$.

While effective, the aforementioned $\bigO(F\log F)$ algorithm for the
evaluation of the Fourier coefficients of $\chi^\K$ is somewhat
intricate, and a detailed exposition of the method lies beyond the
scope of the present primarily theoretical paper. Consequently, most
of the TFF numerical results presented here are based instead on a
simpler algorithm for evaluating the Fourier coefficients of
$\chi^\K$, derived from the divergence theorem and described in
Appendix~\ref{div-thm}. Although this approach incurs a higher
computational cost of $\bigO(F^{3/2})$ operations, its simplicity
makes it particularly convenient for implementation and
experimentation. For reference, however, a limited set of numerical
results illustrating the behavior and performance of the
$\bigO(F\log F)$ algorithm is included in Section~\ref{fft-char}.

A potential alternative to the TFF method for the integration
problem~\eqref{Original_int_2D} is to expand a periodic extension
$\widetilde{f}$ (see Remark~\ref{ext-window}) of the function $f$ in a
Fourier series over the domain $R$, and then multiply the resulting
expansion termwise by the characteristic function $\chi^\K$ before
integrating. The resulting integrals can thus be expressed in terms of integrals of products of $\chi^\K$
with periodic complex exponentials, which are, precisely, the Fourier
coefficients of $\chi^\K$ used in the TFF method. This approach,
however, requires computing the Fourier coefficients of
$\widetilde{f}$ via an FFT, leading to a computational cost of
$\bigO(N\log N)$ operations per function $f$. In contrast, once the
Fourier coefficients of $\chi^\K$ have been precomputed, the TFF
method requires only $\bigO(N)$ operations per function $f$. Indeed,
since, as noted in Remark~\ref{ext-window}, suitable periodic
extensions $\widetilde{f}$ can be constructed at a cost of only
$\bigO(N^{(m-1)/m})$ operations, the overall computational cost is
dominated by the trapezoidal quadrature step. The resulting complexity
is therefore the claimed $\bigO(N)$ per function $f$, with a
relatively small prefactor, and, because no FFT is required for each
new function $f$, the method is particularly well suited to efficient
large-scale parallel implementation. Applying a direct
Fourier-expansion strategy to the convolution problem~\eqref{eq1}, further,
would present additional disadvantages relative to the proposed TFF
convolution strategy. Indeed, in such an algorithm, both the kernel $G$ and the
density $\phi$ would need to be expanded over a periodic domain that
is twice as large in each spatial dimension as the domain used for the
integration problem. In $m$ dimensions, this enlargement would
increase both computational time and memory requirements by a factor
of the order of $2^m$ relative to the TFF convolution algorithm
introduced in this paper.

A version of the TFF convergence proof for one-dimensional $C^1$
functions was presented in~\cite[Appendix A]{hyde2005fast},
establishing an $\bigO(\log F/F^3)$ convergence rate for integration
on the basis of a $2F$-mode truncated Fourier series representation of
the characteristic function.  Theorem~\ref{theorem_main_hd} below
significantly extends that result to piecewise smooth functions
defined over arbitrary $m$-dimensional domains ($m \geq 1$) and, for
example, the new construction and proof yield superalgebraically-fast
convergence, rather than $\bigO(\log F/F^3)$ convergence, for the 1D
problem considered in that contribution.

The remainder of the paper is organized as
follows. Section~\ref{Method} introduces the proposed TFF approach,
and Section~\ref{CA} establishes the convergence properties of the
method. A range of numerical experiments illustrating the 
performance of the TFF method is then provided in 
Section~\ref{Numerics}. Appendices~\ref{div-thm} and~\ref{wind_app} describe the aforementioned divergence theorem-based algorithm for the evaluation of the Fourier coefficients of $\chi^\K$, and the window function used, respectively.

	
\section{Integration via Truncated Fourier Filtering
  (TFF)\label{Method}} We seek to evaluate the
integral~\eqref{Original_int_2D} for a piecewise $C^\infty$ function
$f :\K\to\mathbb{C} $ (Definition~\ref{def_pwCinf} below), where $\K\subset\mathbb{R}^m$ is a compact set
with a possibly non-smooth boundary, of the type described in the
following remark.
\begin{remark}\label{assm:K}
  Throughout this paper it is assumed that, for a certain $P\in\N$,
  there exist mutually disjoint open sets $\{\Omega_{j}\}_{j=1}^{P}$
  contained in the interior $\interior \K$ of $\K$, each with a
  Lipschitz boundary---or, more generally, with a piecewise smooth
  boundary that may fail to be Lipschitz, provided that the set of
  non-smooth boundary points (e.g., corners, edges, or
  lower-dimensional cusps, such as the astroid curve shown in the
  rightmost panel of Figure~\ref{fig:Scatterers_volume}) has Hausdorff
  $(m-1)$-measure zero~\cite[Section 5.8, Theorem
  5.16]{evans2015measure}---such that letting $K_j$ denote the closure
  $\overline{\Omega_{j}}$ of $\Omega_{j}$, $\K$ equals the union of
  the sets $K_j$:
\begin{equation}\label{K-set} \K=\bigcup_{j=1}^{P}K_j,
	\quad\mbox{where}\quad K_j =
	\overline{\Omega_{j}}\quad\mbox{with}\quad \Omega_{j}\cap
	\Omega_{i}=\emptyset \text{ for } j\neq i.
\end{equation}
At various points throughout this paper, the  components $K_j$
of $\mathcal{K}$ in \eqref{K-set} are generically denoted by $K$---so
that, in particular, $K$ denotes a compact set with either Lipschitz
boundary or a piecewise smooth but possibly not Lipschitz boundary for
which the set of non-smooth boundary points has Hausdorff
$(m-1)$-measure zero.
\end{remark}


As indicated in what follows, the proposed TFF method utilizes a
hyper-rectangular domain
\begin{equation}\label{eq:R} R=\prod_{r=1}^{m}[a_r,b_r]
\end{equation} whose interior 
$\interior R$ contains $\K$: $\K\subset\interior R$.  We restrict attention to functions $f$ that are piecewise
$C^\infty$ in the sense of Definition~\eqref{def_pwCinf} below.
\begin{definition}\label{def_pwCinf}\hfill
  \begin{itemize}
  \item[(i)] For a compact set $K\subset\mathbb{R}^m$, a function
$f:K\to\mathbb{C}$ is said to be $C^\infty$ in $K$ (denoted $f\in
C^\infty(K)$) if and only if there exists a function $ g\in
C^\infty(\mathbb{R}^m)$ such that the restriction of $g$ to $K$
equals $f$.
\item[(ii)] A function $f$ defined in a compact set
  $\K\subset\mathbb{R}^m$ is said to be piecewise $C^{\infty}$ in $\K$
  (denoted by $f\in C^\infty_\mathrm{pw}(\K)$) if and only if there exist an integer $P>0$ and sets $\{K_j\}_{j=1}^P$ such that $\K$ satisfies \eqref{K-set}
  and $f|_{K_j}\in C^{\infty}(K_{j})$ for $1\leq j\leq P$.
\end{itemize}
\end{definition} Throughout this paper, we consider infinitely
differentiable $R$-periodic functions; that is, functions
$f \in C^{\infty}(\R^m)$ which 
satisfy
\begin{equation}\label{periodicity}
  f(x_1,\dots,x_{r-1},x_r+\Lambda_r,x_{r+1}\dots,x_m)=f(x_1,\dots,x_{r-1},x_r,x_{r+1}\dots,x_m),\quad \Lambda_r = (b_r - a_r),
\end{equation}
for $1 \leq r \leq m$ and $x\in\R^m$, where $a_r$ and $b_r$ are defined
in~\eqref{eq:R}. The set of all such functions is denoted by
\begin{equation}
  \label{eq:per-def} C_\mathrm{per}^\infty(R)=\left\{ g\in
C^\infty(R)\mbox{ whose $R$-periodic extension to $\mathbb{R}^m$
belongs to $C^\infty(\mathbb{R}^m$)} \right\}.
\end{equation} A function $g\in C_\mathrm{per}^\infty(R)$ will be said
to be a smooth $R$-periodic function.
\begin{remark}\label{extension}
  For any given compact set $K$ as in Remark~\ref{assm:K} with $K$
  contained in $\interior R$ and for any function $f\in C^{\infty}(K)$,
  there exists a function $\widetilde f\in C_\mathrm{per}^\infty(R)$
  whose restriction to $K$ equals $f$---a fact that can be
  established, in view of Definition~\ref{def_pwCinf}, by using an
  infinitely differentiable window function $w$ which equals one on
  $K$ and which vanishes in $\mathbb{R}^m \setminus R$; see
  e.g.~\cite[Proposition 2.25]{lee2013introduction}. (Such window
  functions can be readily constructed numerically using exponential
  functions, the error function, or piecewise polynomials; see Appendix~\ref{wind_app}.)
\end{remark}
\begin{remark}\label{ext-window} In many cases an explicit analytical
  expression of the function $f$ is available which is valid and
  infinitely differentiable throughout $R$, and the strategy outlined
  in Remark~\ref{extension} can be easily applied. However, such
  expressions may not be available in certain cases---for example,
  when the function $f$ is known only in discrete form and only over
  the set $K$, as is often the case for functions arising from the
  numerical solution of a partial differential equation (PDE), perhaps
  obtained through iterative methods or time-stepping schemes. In the
  latter cases smooth numerical periodic extensions $\widetilde f$ may
  be obtained by employing certain elements, namely the
  extension-to-zero stages, of the extension
  methods~\cite{amlani2016fc,albin2011spectral}
  and~\cite{bruno2022two,bruno_yang2025fc2d} in dimensions $m=1$ and
  $m> 1$, respectively. Importantly, these extension-to-zero
  algorithmic stages run at computing costs of $\bigO(N^{(m-1)/m})$
  operations in the $m$-dimensional context, and are therefore
  computationally inexpensive.
\end{remark}
\begin{remark}\label{extension2} If the function $f \in C^{\infty}(K)$
in~\eqref{Original_int_2D} is given by an expression that defines an
$R$-periodic function, then the corresponding TFF {\em integration}
algorithm does not require the window-function multiplication step
described in Remark~\ref{extension}. In such cases, the given
$R$-periodic function $f$ can be used directly---which is preferable,
as it avoids the increase in the size of the underlying derivatives
associated with windowing. However,  the TFF {\em convolution} algorithm does
require that the density function $\phi$ in~\eqref{eq1} be multiplied
by a window function as described in Remark~\ref{extension}.
\end{remark}
\subsection{TFF approximation for $g\in C^{\infty}(K)$}\label{TFF}

Clearly, for a given compact set $K\subset \interior R$ as in
Remark~\ref{assm:K} above, and for any function $g\in C^\infty(K)$,
using an extension $\widetilde{g}\in C_\mathrm{per}^\infty(R)$ of $g$
(see Remark~\ref{extension}), the integral~\eqref{Original_int_2D} may
be re-expressed in the form
	\begin{equation}\label{Original_int_2D_D} \A[g]=\int_{R}
\widetilde{g}(\x) \chi^{K}(\x) d\x,
\end{equation} where $\chi^{K}$ denotes the characteristic function of
$K$.  The proposed TFF method for a function $g\in C^{\infty}(K)$
proceeds by substituting the characteristic function $\chi^{K}$
in~\eqref{Original_int_2D_D}, viewed as an $R$-periodic function
(eq.~\eqref{periodicity}),  by a truncated version 
\begin{equation}\label{sec:eq2}
\chi^{K}_{F}(\x)=\sum_{\ell_{m}=-F_m}^{F_m}\cdots
\sum_{\ell_{2}=-F_2}^{F_2} \sum_{\ell_{1}=-F_1}^{F_1} c_{\ell} e^{2\pi
\ci \sum_{r=1}^{m}\frac{\ell_{r}x_{r}}{\Lambda_r} },
\end{equation}
(up to a given order $F=(F_1,\dots,F_m)$) of the Fourier expansion of
$\chi^{K}$ over the rectangle $R$, and then evaluating the resulting
integral by means of the trapezoidal rule on the rectangle $R$.

It is important to note that the accuracy of the TFF algorithm depends
crucially on the corresponding accuracy of the Fourier coefficients
$c_{\ell} $ ($\ell=(\ell_{1},\ell_{2},\cdots, \ell_{m})$), which are
given by
\begin{equation}\label{sec2:eq3} c_{\ell} = \left(\prod_{r=1}^{m}
\frac{1}{\Lambda_r} \right) \int_R\chi^{K}(\x) e^{-2\pi \ci
\sum_{r=1}^{m}\frac{\ell_{r}x_{r}}{\Lambda_r} } ~d\x.
\end{equation}
In detail, for $g\in C^\infty(K)$, the proposed strategy produces
numerical values of $\A[g]$ via the $N$-point approximation
\begin{equation}\label{Original_int_2D_FS} \A_{n,F} [g]= 
 T^{m,R}_n\left[ \chi^{K}_F\widetilde g \right], \qquad n=(n_1, \dots,
n_m), \qquad N=\prod_{r=1}^m (n_r+1),
\end{equation}
for given positive integers $n_1, \dots, n_m$, where $T^{m,R}_n$
represents the $m$-dimensional trapezoidal-rule quadrature operator on
the hyper-rectangle $R$; see Remark~\ref{ext-dep}. This operator is thus
constructed through $m$-fold sequential application of
the one-dimensional trapezoidal quadrature rule, which, for a function
$\beta: [a_0,b_0] \to \mathbb{C}$ discretized using $q+1$ points, is
defined as
\begin{equation}\label{sec2:eq5} T^{1,[a_0,b_0]}_{q}[\beta] = \frac{h_0}{2}
\big(\beta(a_0) + \beta(b_0)\big) + h_0\sum_{k=1}^{q-1} \beta(a_0 + k h_0),
\end{equation} where $h_0 = (b_0-a_0)/q$.

Clearly, letting $h_r = \frac{b_r - a_r}{n_r}$ for $1 \leq r \leq m$,
setting $h = (h_1, h_2, \dots, h_m)\in\mathbb{R}^m$,
$k = (k_1, k_2, \dots, k_m)\in\left(\mathbb{N}\cup
\{0\}\right)^m$, and introducing the componentwise inequality between
$m$-dimensional vectors $k, \ell \in \mathbb{Z}^m$,
\begin{equation}
  \label{eq:compwise}
  k \leq \ell\quad\mbox{if and only if}\quad k_r \leq \ell_r\quad \mbox{for}\quad 1 \leq r \leq m,
\end{equation}
for a given a
function $\gamma: R \to \mathbb{C}$, the quadrature expression
$T^{m,R}_n[\gamma]$ utilizes the values of the function $\gamma$ on the
$N$-point Cartesian mesh
\begin{equation}
  \label{eq:disc-pts}
  \mathcal{D}_h = \left\{ \x_k = (x_{k_1}^1, x_{k_2}^2, \dots, x_{k_m}^m)\ \middle|\ x_{k_r}^r = a_r + k_r h_r,\ 0 \leq k \leq n,\ 1 \leq r \leq m \right\},
\end{equation}
within the $m$-dimensional hyper-rectangle~\eqref{eq:R}. Note, for
reference, that an application of the trapezoidal quadrature rule to
the original integral~\eqref{Original_int_2D} yields a degraded order
of convergence---of order $\mathcal{O}(h^1)$ in 1D and rather erratic
convergence in $m$D with $m\geq 2$---for the types of discontinuous
functions under consideration; cf. Remark~\ref{trap-1st-ord}. But, as
shown in Section~\ref{CA}, the approximation~\eqref{Original_int_2D_FS}
converges {\em to the original integral~\eqref{Original_int_2D}} with
superalgebraically small absolute-value errors $E$:
\begin{equation}\label{super-alg} E \leq \mathcal{O}((h_\mathrm{max})^{-p}) \qquad
  \text{for all $p\in \N$} \qquad (h_\mathrm{max} =\max_{1\leq r\leq m} h_r),
\end{equation} for functions $g$ that are smooth over $K$, provided,
1)~For some
constant $a$, $0<a<1$, we have $F_r\leq an_r$ for
$1\leq r\leq m$ as $h_\mathrm{max}\to 0$; and,
2)~The Fourier coefficients defined in~\eqref{sec2:eq3} are known accurately, either analytically or,
otherwise, by means of an accurate numerical procedure---in which case
the convergence~\eqref{super-alg} is limited by the error in the
numerical Fourier coefficients.  Note that, of course, the function
values of $\chi_{F}^{K}$ on the equispaced grid $\x_{\bk}$, which are
needed in \eqref{Original_int_2D_FS}, can be computed in
$\bigO(N \log N)$ operations by applying the FFT algorithm to the
Fourier coefficients $c_\ell$.
\begin{remark}\label{ext-dep}
  Although the notation does not make it explicit, the discrete
  operator $\A_{n,F}$ in~\eqref{Original_int_2D_FS} depends on the
  chosen extension operator $g \mapsto \widetilde{g}$ (see
  Remarks~\ref{extension} and~\ref{ext-window}). However,
  Theorem~\ref{theorem_main_hd}, which applies in the broader setting
  of piecewise smooth functions $f$, ensures that, irrespective of the
  smooth extension selected, the TFF
  approximation~\eqref{Original_int_2D_FS} converges to the exact
  value $\A[g]$ at a superalgebraic rate.
  \end{remark}
	

\subsection{TFF approximation for $f\in
C_{\mathrm{pw}}^{\infty}(\K)$}\label{TFF-pw} The TFF algorithm
proposed in the previous section can easily be extended to functions
$f \in C_{\mathrm{pw}}^{\infty}(\K)$ for a given compact set $\K$. To
do this we first recall that, per Definition~\ref{def_pwCinf}, for a
given function $f\in C_{\mathrm{pw}}^{\infty}(\K)$ there exist
$P\in\N$ and disjoint open sets $\Omega_j$, $1\leq j\leq P$, such that
$\K$ is given by equation~\eqref{K-set} and
$f|_{K_j}\in C^{\infty}(K_{j})$ for $1\leq j\leq P$. Calling
$f_j=f|_{K_j}$ for notational simplicity, we let $\widetilde{f_{j}}$
denote a smooth and $R$-periodic extension of $f_{j}$ for each
$1\leq j\leq P$; see Remark~\ref{extension}. Further, we let
$\chi^{K_j}$ denote the characteristic function of $K_{j}$ and, for
$F\in\N^m$, we let $\chi_{F}^{K_{j}}$ denote the truncated Fourier
expansion of $\chi^{K_{j}}$ of order $F$,
\begin{equation}\label{sec:eq4}
\chi^{K_j}_{F}(\x)=\sum_{\ell_{m}=-F_m}^{F_m}\cdots
\sum_{\ell_{2}=-F_2}^{F_2} \sum_{\ell_{1}=-F_1}^{F_1} c_{j,\ell}
e^{2\pi \ci \sum_{r=1}^{m}\frac{\ell_{r}x_{r}}{\Lambda_r} },
\end{equation}
where $ c_{j,\ell}$ is given by~\eqref{sec2:eq3} with $K$ substituted
by $K_j$.  We clearly have
\begin{equation}\label{int_pw} \A[f]=\int_{\K} f(\x) d\x =
\sum_{j=1}^{P} \int_{K_j} f_{j}(\x) d\x = \sum_{j=1}^{P} \int_{R}
\widetilde{f_{j}}(\x) \chi^{K_{j}}d\x.
\end{equation} We may then use the TFF approximation given in
equation~\eqref{Original_int_2D_FS} with $K = K_j$ for $1\leq j\leq
P$, and we thus obtain the following TFF approximation for the
function $f \in C_{\mathrm{pw}}^{\infty}(K)$:
\begin{equation}\label{TFF_pw} \A_{n,F}\ftilde = T^{m,R}_n\left[
\sum_{j=1}^{P}\chi_{F}^{K_j} \widetilde f_{j} \right], \qquad n=(n_1,
\dots, n_m).
\end{equation}

\section{Convergence Analysis}\label{CA}
This section presents theoretical estimates for the error incurred in
the evaluation of the integral~\eqref{Original_int_2D} by the TFF
method, assuming that $f \in C_\mathrm{pw}^{\infty}(\K)$ and that $\K$
is a compact set of the form~\eqref{K-set}, where the boundary
$\partial K_j$ of each component $K_j$ is either Lipschitz, or
piecewise smooth but possibly not Lipschitz.  In what follows the
notation $\|\ell\|_\infty = \max_{r=1,\dots,m}\{|\ell_r|\}$ is used,
and, for notational simplicity we assume that the hyper-rectangle $R$
is given by
\begin{equation}
  \label{eq:Rpi}
  R=[-\pi,\pi]^{m}.  
\end{equation}

The required error estimate is presented in
Theorem~\ref{theorem_main_hd} below, following a sequence of three
preliminary lemmas.  The first lemma provides estimates on the size of
the Fourier coefficients of a smooth periodic function
$\widetilde{g}\in C^{\infty}_{\mathrm{per}}(R)$.
\begin{lemma}\label{lem_FC_hd_Cinf} Let
  $\ell=(\ell_1,\ell_2,\dots,\ell_m)\in\Z^{m}\setminus\{0\}$ and let
  $\widetilde{g}\in C^{\infty}_{\mathrm{per}}(R)$. Then, the $\ell$-th Fourier coefficient of $\widetilde{g}$
  tends to zero superalgebraically fast as $\ell\to \infty$, that is to say, for all $q = (q_1,q_2,\dots,q_m)\in\N^m$ we have
	\begin{equation} | d_{\ell} | \leq C \prod_{r=1}^m(|\ell_r|+1)^{-q_r} \quad \mbox{as} \quad \|\ell\|_\infty \to\infty,
	\end{equation}
        for some constant $C$.
\end{lemma}
\begin{proof} Integrating by parts the expression
  \begin{equation}\label{dell} d_{\ell} = \frac{1}{(2\pi)^m} \int_{R} \widetilde{g}(\x)
    e^{-\ci\ell\cdot\x} d \x,
  \end{equation}
  $q_r$ times with respect to $x_r$ for all $r=1,\dots, m$ such that
  $\ell_r\ne 0$ yields the relation
  \begin{equation}\label{parts}
    d_{\ell} = \frac{1}{(2\pi)^m}
   \prod_{\substack{r=1\\\ell_r\neq 0}}^m (-i\ell_r)^{-q_r}\int_R \frac{\partial^{q_1}}{\partial
      x_1^{q_1}}\frac{\partial^{q_2}}{\partial
      x_2^{q_2}}\cdots\frac{\partial^{q_m}}{\partial
      x_m^{q_m}}\widetilde{g}(\x) e^{-\ci\ell\cdot\x} d\x.
  \end{equation}
  The proof then follows by using a uniform bound for the integrals in~\eqref{parts} together with the estimate
  \[
    \left| (-i\ell_r)^{-q_r}\right| = \mathcal{O}\left(
      (|\ell_r|+1)^{-q_r}\right)\quad\mbox{as}\quad \ell_r\to\infty.
  \]
\end{proof}

Let $g\in C^{\infty}(K)$ and let $\widetilde{g}$ be a smooth and
periodic extension of $g$ to the rectangle $R$:
$\widetilde{g}\in C^{\infty}_{\mathrm{per}}(R)$ and
$\widetilde{g}(\x) = g(\x)$ for $\x\in K$. Introducing the operator
\begin{equation}\label{Original_int_2D_FS-cont} 
	\A_{F}[\widetilde{g}]=\int_{R}
	\widetilde{g}(\x) \chi_{F}^{K}(\x) d\x,
\end{equation} 
for given $F=(F_1,F_2,\dots,F_m)\in \N^{m}$ (see equation~\eqref{sec:eq2}), the total error
$E_{\mathrm{tot}}^{n}(F,\widetilde{g})$ in the TFF approximation of the integral
$\A[g]$ is estimated as follows:
\begin{equation}\label{eq_tot_error_hd}
  E_{\mathrm{tot}}^{n}(F,\widetilde{g}) = \left|\A[g]-\A_{n,F}[\widetilde{g}]\right| = 
  \left|\int_{K}g(\x) d\x -
    T_{n}^{m,R}\left[\chi_{F}^{K}\widetilde{g}\right]\right| \leq
  E_{\mathrm{trunc}}(F,\widetilde{g}) + E_{\mathrm{trap}}^{n}(F,\widetilde{g}),
\end{equation} where
\begin{align}
  \label{eq_trun_error_hd} E_{\mathrm{trunc}}(F,\widetilde{g}) &= \left|\A[g]-\A_{F}[\widetilde{g}]\right| =\left|\int_{R}\widetilde{g}(\x)\chi^{K}(\x) d\x -\int_{R}\widetilde{g}(\x)\chi^{K}_{F}(\x) d\x \right|,\text{and}\\\label{eq_trap_error_hd}
  E_{\mathrm{trap}}^{n}(F,\widetilde{g}) &= \left|\A_{F}[\widetilde{g}]-\A_{n,F}[\widetilde{g}]\right|=\left|\int_{R}\widetilde{g}(\x)\chi^{K}_{F}(\x) d\x -T_{n}^{m,R}\left[\chi^{K}_{F}\widetilde{g}\right]\right|.
\end{align}
The following two lemmas provide estimates for the error terms
in~\eqref{eq_trun_error_hd} and~\eqref{eq_trap_error_hd}.
\begin{lemma}\label{lemma_err_trunc_hd} 
  Assume equation~\eqref{eq:Rpi} holds, (cf.~\eqref{eq:R}), let $K$
  denote a set as in Remark~\ref{assm:K}, let $g\in C^{\infty}(K)$,
  and call $\widetilde{g}\in C_\mathrm{per}^{\infty}(R)$ an extension
  of $g$. Further, let $c_{\ell}$ and $d_{\ell}$ denote the $\ell$-th
  Fourier coefficients of the $R$-periodic functions $\chi^{K}$ and
  $\widetilde{g}$ respectively (equations~\eqref{sec2:eq3}
  and~\eqref{dell} with $a_{r}=-\pi$ and $b_{r} = \pi$).  Then, for
  any positive integer $p$ there is a constant $C$ such that
  \begin{equation} \label{tr_bnd}
  	E_{\mathrm{trunc}}(F,\widetilde{g}) \leq
	\frac{C}{\left(F_{\mathrm{min}}\right)^{p}} \quad\mbox{where}\quad F_{\mathrm{min}} = \min_{1\leq r\leq m}\{F_r\}.
 \end{equation} 
\end{lemma}
 \begin{proof} We have
   \begin{align*}
     E_{\mathrm{trunc}}(F,\widetilde{g})  &= \left|\int_{R}\widetilde{g}(\x)\chi^{K}(\x) d\x -
     \int_{R}\widetilde{g}(\x)\chi^{K}_{F}(\x) d\x \right|\\ 
     &= (2\pi)^{m}\left| \sum_{\ell\in \Z^{m}} c_{\ell}d_{-\ell} - \sum_{\|\ell\|_{\infty}\leq F } c_{\ell}d_{-\ell}\right|\leq  (2\pi)^{m} \sum_{\|\ell\|_{\infty}>F} |c_{\ell}|
     |d_{-\ell}|.
 \end{align*} 
 Calling $\gamma = \gamma(\x)$ the integrand in~\eqref{sec2:eq3},
 letting $\mathrm{e}_r\in\R^m$ denote the $r$-th canonical-basis
 vector with $r$ such that $|\ell_r|=\|\ell\|_\infty$, and applying
 the divergence theorem (for either Lipschitz domains or piece-wise
 smooth but possibly non-Lipschitz domains~\cite[Section 5.8, Theorem
 5.16]{evans2015measure}) to the integral over $K$ of
 $\mathrm{div}(\gamma \mathrm{e}_r)$ shows that
 \begin{equation}\label{eq_decay_cl} 
   |c_{\ell}|\leq C/(|\ell_r|+1)
  \quad \mbox{for}\quad 1\leq r\leq m \quad\mbox{for all}\quad  \ell\in\Z^m
\end{equation}
for some constant $C$.
Using~\eqref{eq_decay_cl} together with Lemma~\ref{lem_FC_hd_Cinf}, for
some constant $C_0$ and for any $q = (q_1,\dots,q_m)$ we may write
\begin{align} 
  E_{\mathrm{trunc}}(F,\widetilde{g}) &\leq (2\pi)^{m}
                            \sum_{r=1}^m\sum\limits_{\substack{\ell\in\mathbb{\Z}^{m}\\|\ell_{r}|>F_r}} |c_{\ell}| |d_{-\ell}|\\\label{prod-sum}
                          &\leq C_0 \sum_{r=1}^m\sum\limits_{\substack{\ell\in\mathbb{\Z}^{m}\\|\ell_{r}|>F_r}} 
  \frac{1}{(|\ell_r|+1)} \frac{1}{(|\ell_r|+1)^{q_r}} 
  \prod_{\substack{s=1\\s\neq r}}^{m}\frac{1}{(|\ell_{s}|+1)^{q_s}} \\\label{prod-sum-3}
                          &\leq C_0 \sum_{r=1}^m\sum\limits_{\substack{\ell_r\in\mathbb{\Z}\\|\ell_{r}|>F_r}} 
   \frac{1}{(|\ell_r|+1)^{q_r+1}} S_r                  
\end{align}
where letting
$\overline{\ell_{r}}=(\ell_1,\dots,\ell_{r-1},\ell_{r+1},\dots,\ell_m)\in\Z^{m-1}$
we have set
\begin{equation}
  \label{eq:mm1sum}
  S_r=\sum\limits_{\substack{\overline{\ell}_r\in\mathbb{\Z}^{m-1}}} \prod_{\substack{s=1\\s\neq r}}^{m}\frac{1}{(|\ell_{s}|+1)^{q_s}}.
\end{equation}
Clearly we have
\begin{align*}
S_r=\sum\limits_{\ell_1\in\mathbb{\Z}} \frac{1}{(|\ell_{1}|+1)^{q_{1}}}\dots\sum\limits_{\ell_{r-1}\in\mathbb{\Z}} \frac{1}{(|\ell_{r-1}|+1)^{q_{r-1}}} \sum\limits_{\ell_{r+1}\in\mathbb{\Z}} \frac{1}{(|\ell_{r+1}|+1)^{q_{r+1}}}\dots \sum\limits_{\ell_{m}\in\mathbb{\Z}} \frac{1}{(|\ell_{m}|+1)^{q_{m}}}, 
\end{align*}
and, thus, taking, as we may, $q$ such that $q_r\geq 2$ for $s\ne r$,
we see that the series $S_r$ is convergent for $1\leq r\leq m$.
Since, additionally, the inner sum in~\eqref{prod-sum-3} is bounded by
a constant multiple of $1/F_r^{q_r}$, taking $q_r\geq p$ for
$1\leq r\leq m$, the bound~\eqref{tr_bnd} follows and the proof is
complete.
\end{proof}
 
\begin{lemma}\label{lemma_err_trap_hd} Let $0<a<1$
  and $g\in C^{\infty}(K)$ (with $K\subset R$ a set as in
  Remark~\ref{assm:K}) be given, and let $\widetilde{g}$ denote a
  smooth and $R$-periodic extension of $g$. Let $F\in\mathbb{N}^{m}$
  and $n\in\mathbb{N}^{m}$ be such that $F_r\leq a n_r$ for
  $1\leq r\leq m$. Then, for any integer $p\geq 2$ there is a constant
  $C$ such that the trapezoidal-rule error~\eqref{eq_trap_error_hd}
  satisfies
   \begin{equation}\label{eq:trap-err} 
  	E_{\mathrm{trap}}^{n} (F,\widetilde{g}) \leq
   C\, \frac{\log(F_{\mathrm{max}})}{\left(F_{\mathrm{min}}\right)^{p}}, \quad\mbox{where} \quad F_{\mathrm{max}} = \max_{1\leq r\leq m}\{F_r\}\quad \mbox{and}\quad F_{\mathrm{min}} = \min_{1\leq r\leq m}\{F_r\}.
   \end{equation} 
 \end{lemma}
 \begin{proof} 
   In view of the smoothness and periodicity assumed on the function
   $\widetilde{g}$, the Fourier series expansion
   \begin{equation}\label{series-f-tilde}
     \widetilde{g}(\x) = \sum\limits_{k\in\Z^m} d_{k}
     e^{\ci k \cdot\x}
\end{equation} 
converges uniformly in $R$. Thus, defining the componentwise absolute
value of a vector $\ell \in \mathbb{Z}^m$ as
\[
|\ell|_\mathrm{c} = (|\ell_1|, |\ell_2|, \dots, |\ell_m|)
\]  
and using~\eqref{sec:eq2} and~\eqref{series-f-tilde} together
with~\eqref{eq:compwise} we obtain
\begin{equation}\label{eq_trap-md} 
 	T_{n}^{m,R}\left[\chi^{K}_{F}\widetilde{g}\right] = T_{n}^{m,R}\Bigg [\sum\limits_{\substack{\ell\in\Z^m\\|\ell|_\mathrm{c}\leq
F}}\sum\limits_{k\in\Z^{m}}
 	c_{\ell}d_{k}e^{\ci(k+\ell)\cdot \x}\Bigg ]=\sum\limits_{\substack{\ell\in\Z^m\\|\ell|_\mathrm{c}\leq
F}}\sum\limits_{k\in\Z^m}
c_{\ell}d_{k} T_{n}^{m,R}\left[e^{\ci(k+\ell)\cdot \x}\right].
 \end{equation} 
 Since
 \begin{equation*} 
 T_{n}^{m,R}\left[e^{\ci(k+\ell)\cdot\x}\right] = 
 T_{n_1}^{1,[-\pi,\pi]}\left[e^{\ci(k_1+\ell_1)x_1}\right]
 	T_{n_2}^{1,[-\pi,\pi]}\left[e^{\ci(k_2+\ell_2)x_2}\right]\cdots
T_{n_m}^{1,[-\pi,\pi]}\left[e^{\ci(k_m+\ell_m)x_m}\right],
 \end{equation*} 
  using the relation 
 \begin{equation} 
 	T_{n_r}^{1,[-\pi,\pi]}\left[e^{\ci qx}\right] = \begin{cases} 2\pi &\text{ if } q=
sn_r, \text{ for some } s\in \Z \\ 0 &\text{ otherwise},
 	\end{cases}\quad 1\leq r\leq m,
 \end{equation} 
 we conclude that the only non-vanishing terms in~\eqref{eq_trap-md}
 are those for which the  integer vectors $k+\ell$ satisfy the conditions
 $k_1+\ell_1=j_1n_1, k_2+\ell_2=j_2n_2, \dots, k_m+\ell_m=j_mn_m$ for
 some vector $j=(j_1,j_2,\dots,j_m)\in\mathbb{Z}^m$. Thus,
 introducing the notation
\begin{equation}\label{eq_prod-vect}
jn = (j_1n_1,\dots,j_m n_m)
\end{equation}
in the case $k+\ell = jn$ we obtain
$T_{n}^{m,R}\left[e^{\ci(k+\ell)\cdot\x}\right] =
T_{n}^{m,R}\left[e^{\ci\, jn\cdot\x}\right]=(2\pi)^m$. Using these
notations and identities we see that
\begin{equation}\label{eq_trap_sum} 
	T_{n}^{m,R}\left[\chi^{K}_{F}\widetilde{g}\right] =
	(2\pi)^m \sum\limits_{\substack{\ell\in\Z^m\\|\ell|_\mathrm{c}\leq
F}}\sum_{j\in\Z^m}d_{jn-\ell}c_{\ell} = (2\pi)^m \sum_{\substack{\ell\in\Z^m\\ |\ell|_\mathrm{c}\leq
F}}d_{-\ell}c_{\ell} + (2\pi)^m \sum\limits_{\substack{\ell\in\Z^m\\|\ell|_\mathrm{c}\leq
F}}\sum_{j\in\Z^m\setminus\{0\}}d_{jn-\ell}c_{\ell}.
\end{equation}

Since the integral of an $R$-periodic function $g$ equals $(2\pi)^m$
times the zeroth-order Fourier coefficient of $g$ (which, for the
function $\chi_{K}^{F}\widetilde{g}$ in~\eqref{eq_trap_sum} coincides
with the first sum on the right-hand side of that equation), it
follows that
\begin{align*} 
E_{\mathrm{trap}}^{n}(F,\widetilde{g}) =  \Bigg|T_{n}^{m,R}\left[\chi^{K}_{F}\widetilde{g}\right] -
  \int_{R}\chi^{K}_F(\x)\widetilde{g}(\x)d\x \Bigg|
  &\leq (2\pi)^m\sum\limits_{\substack{\ell\in\Z^m\\|\ell|_\mathrm{c}\leq
F}}\sum\limits_{j\in\Z^m\setminus\{0\}}
   \left|d_{jn-\ell}c_{\ell} \right|.
\end{align*} 
Using Lemma~\ref{lem_FC_hd_Cinf}, for any $(p_1,p_2,\dots,p_m)\in\N^m$ we
then obtain
\begin{align}\label{eq_trap-er-sum} 
 E_{\mathrm{trap}}^{n}(F,\widetilde{g})&\leq (2\pi)^mC\sum\limits_{\substack{\ell\in\Z^m\\|\ell|_\mathrm{c}\leq
F}}\sum\limits_{j\in\Z^m\setminus\{0\}}
   |c_{\ell}|\prod_{r=1}^m\frac{1}{(|j_rn_r-\ell_r|+1)^{p_r}}.
\end{align} 

In order to estimate the quantity on the right-hand side
of~\eqref{eq_trap-er-sum} we define the equivalence relation $\sim$ on
$\Z^m$ under which, for $j,\,k\in\Z^m$, $j\sim k$ if and only if the
set of $r$ for which $j_r=0$ coincides with the set of $r$ for which
$k_r=0$. Clearly, the set of equivalence classes is in a one-to-one
correspondence with the set $\Z_2^{m}=\{0,1\}^m\subset\Z^m$.  On the
basis of this definition we may re-express~\eqref{eq_trap-er-sum} in
the form
\begin{equation}\label{eq_trap-er-sum2}
E_{\mathrm{trap}}^{n}(F,\widetilde{g})\leq (2\pi)^mC\sum_{\substack{i\in\Z_2^m\\ i\neq 0}}\sum\limits_{\substack{\ell\in\Z^m\\|\ell|_\mathrm{c}\leq
F}}|c_{\ell}|\sum_{\substack{j\in\Z^m\\j\sim i}}
	 \prod_{r=1}^m\frac{1}{(|j_rn_r-\ell_r|+1)^{p_r}}.
\end{equation}
We now note that, for a given $i \in \Z_2^m \setminus {0}$, the
innermost sum on the right-hand side of~\eqref{eq_trap-er-sum2} can be
re-expressed by decomposing the product within the sum as $A \cdot B$,
where $A$ and $B$ denote the factors corresponding to indices $r$ for
which $i_r = 0$ and $i_r \ne 0$, respectively. Thus, for each fixed
$i \in \Z_2^m \setminus {0}$ this yields
\begin{equation}\label{eq:inner_trap}
  \begin{split}
    \sum_{\substack{j\in\Z^m\\j\sim i}}
    \prod_{r=1}^m\frac{1}{(|j_rn_r-\ell_r|+1)^{p_r}} &= \sum_{\substack{j\in\Z^m\\j\sim i}}
    \prod_{\substack{s=1\\i_s\neq 0}}^m\frac{1}{(|j_sn_s-\ell_s|+1)^{p_s}} 
    \prod_{\substack{t=1\\i_t= 0}}^m\frac{1}{(|\ell_t|+1)^{p_t}}\\ 
                                                     &= \left(\prod_{\substack{s=1\\i_s\neq 0}}^m\frac{1}{n_s^{p_s}}\right) \left(\prod_{\substack{t=1\\i_t= 0}}^m\frac{1}{(|\ell_t|+1)^{p_t}}\right)
    \sum_{\substack{j\in\Z^m\\j\sim i}}
    \prod_{\substack{s=1\\i_s\neq 0}}^m\frac{1}{(|j_s-\frac{\ell_s}{n_s}|+\frac{1}{n_s})^{p_s}},
  \end{split}
\end{equation}
where the last equality results by taking common factors $1/n_s^{p_s}$
and $1/(|\ell_t|+1)^{p_t}$. Since, per~\eqref{eq_trap-er-sum2},
$|\ell_s| \leq F_s$ and since by hypothesis $F_s \leq a n_s$,
it follows that $\frac{|\ell_s|}{n_s} \leq a$.  Assuming, as
we may, that
\begin{equation}\label{eq:ps}
  p_s \geq 2 \mbox{ for } 1 \leq s \leq m,
\end{equation}
it follows that the sum on the right hand side of
\eqref{eq:inner_trap} admits the bound
\begin{equation}\label{eq:innermost}
0 \leq \sum_{\substack{j\in\Z^m\\j\sim i}}
   \prod_{\substack{s=1\\i_s\neq 0}}^m\frac{1}{(|j_s-\frac{\ell_s}{n_s}|+\frac{1}{n_s})^{p_s}}  
   \leq  \sum_{\substack{j\in\Z^m\\j\sim i}}
   \prod_{\substack{s=1\\i_s\neq 0}}^m\frac{1}{(|j_s|-a)^{p_s}} = \prod_{\substack{s=1\\i_s\neq 0}}^m\sum_{\substack{j_s\in\Z\\j\sim i}}
  \frac{1}{(|j_s|-a)^{p_s}} \leq C_1
\end{equation}
for some constant $C_1$ independent of $i$---since, in view of the
assumption \eqref{eq:ps}, each series in the
right-hand expression in~\eqref{eq:innermost} is
convergent for all $i \in \Z_2^m$. Combining~\eqref{eq:inner_trap}
and~~\eqref{eq:innermost} we obtain
\begin{equation}\label{eq:inner_trap2}
 0 \leq  \sum_{\substack{j\in\Z^m\\j\sim i}}
  \prod_{r=1}^m\frac{1}{(|j_rn_r-\ell_r|+1)^{p_r}}\leq C_1\left(\prod_{\substack{s=1\\i_s\neq 0}}^m\frac{1}{n_s^{p_s}}\right)\left( \prod_{\substack{t=1\\i_t= 0}}^m\frac{1}{(|\ell_t|+1)^{p_t}}\right).
\end{equation}
Substituting \eqref{eq:inner_trap2} in
\eqref{eq_trap-er-sum2}, selecting
 \begin{equation}\label{ri}
  \mbox{``for each}~i\in \Z_2^m\setminus\{0\}, ~r=r(i)\in \{1,2,\cdots,m\}~\mbox{such that}~i_{r}\ne 0,\mbox{''}
\end{equation}
and using the estimate~\eqref{eq_decay_cl} with $r = r(i)$, we obtain
\begin{equation}\label{eq:trap3}
  \begin{split}
    E_{\mathrm{trap}}^{n}(F,\widetilde{g})\leq & C_2\sum_{\substack{i\in\Z_2^m\\ i\neq 0}}\left(\prod_{\substack{s=1\\i_s\neq 0}}^m\frac{1}{n_s^{p_s}}\right)\sum\limits_{\substack{\ell\in\Z^m\\|\ell|_\mathrm{c}\leq
        		F}}|c_{\ell}|\prod_{\substack{t=1\\i_t= 0}}^m\frac{1}{(|\ell_t|+1)^{p_t}}\\ \leq & C_3\sum_{\substack{i\in\Z_2^m\\ i\neq 0}}\left(\prod_{\substack{s=1\\i_s\neq 0}}^m\frac{1}{n_s^{p_s}}\right)\sum\limits_{\substack{\ell\in\Z^m\\|\ell|_\mathrm{c} \leq
    F}}\frac{1}{(|\ell_{r(i)}|+1)}\prod_{\substack{t=1\\i_t= 0}}^m\frac{1}{(|\ell_t|+1)^{p_t}},
    \end{split}
\end{equation}
for certain positive constants $C_2$ and $C_3$.  In order to proceed
with trapezoidal-rule error estimation, in view of~\eqref{ri}
and~\eqref{eq:trap3} we now define
\begin{equation}\label{eq:ranges}
  a_{i,t}(k)= \begin{cases}
    \frac{1}{|k|+1} &\mbox {if }i_t\ne 0,~ t=r(i),\\
     1 &\mbox {if } ~i_t\ne 0,~t\ne r(i),\\
    \frac{1}{(|k|+1)^{p_t}} &\mbox {if } i_t= 0,\\
  \end{cases}
\qquad   (i\in\mathbb{Z}_2^m,~ 1\leq t\leq m,~ -F_t\leq k\leq F_t)  
\end{equation}
and we re-write the second estimate in~\eqref{eq:trap3} in the form
\begin{equation}\label{r-of-i}
  E_{\mathrm{trap}}^{n}(F,\widetilde{g})\leq C_3 \sum_{\substack{i\in\Z_2^m\\ i\neq 0}}\left(\prod_{\substack{s=1\\i_s\neq 0}}^m\frac{1}{n_s^{p_s}}\right)\sum\limits_{\substack{\ell\in\Z^m\\|\ell|_\mathrm{c}\leq
      F}}\prod_{\substack{t=1}}^m a_{i,t}(\ell_t).
\end{equation}                 

Re-expressing the sum over $\ell$ into $m$ one-dimensional
sums over each $\ell_t$ ($1 \leq t \leq m$) we now obtain
\begin{equation}\label{eq:trap-3}
  E_{\mathrm{trap}}^{n}(F,\widetilde{g})\leq C_3 \sum_{\substack{i\in\Z_2^m\\ i\neq0}}
		\left(\prod_{\substack{s=1\\i_s\neq0}}^m\frac{1}{n_s^{p_s}}\right)
  \prod_{\substack{t=1}}^{m} \sum_{\ell_t=-F_t}^{F_t}a_{i,t}(\ell_t).
\end{equation}
Now, the second product may be decomposed in terms of the three ranges
of $t$ considered in~\eqref{eq:ranges}. In order to bound each one of
the corresponding products, and recalling the assumption \eqref{eq:ps}, we note
that, for some constant $\widetilde{C}$, the bounds
\begin{equation*}
\sum_{\ell_{r(i)}=-F_{r(i)}}^{F_{r(i)}}\frac{1}{|\ell_{r(i)}|+1} \leq \widetilde{C} \log(F_{r(i)}),\quad \sum_{\ell_{t}=-F_{t}}^{F_{t}}1 \leq \widetilde{C}\, F_t,\quad\mbox{and}\quad\sum_{\ell_{t}=-F_{t}}^{F_{t}}\frac{1}{(|\ell_t|+1)^{p_t}} \leq \widetilde{C},
\end{equation*}
hold for all $F\in\mathbb{N}^m$. Substituting these bounds
in~\eqref{eq:trap-3} we then obtain
\begin{equation}\label{eq:trap4}
  E_{\mathrm{trap}}^{n}(F,\widetilde{g})\leq C_4 \sum_{\substack{i\in\Z_2^m\\ i\neq0}}
 \left( \prod_{\substack{s=1\\i_s\neq0}}^m\frac{1}{n_s^{p_s}}\right) (\log(F_{r(i)}))
 \left( \prod_{\substack{t=1\\i_t\neq 0,t\ne r(i)}}^{m} F_t\right),
\end{equation}
for some constant $C_4 > 0$.
Noting that
\begin{equation}
  \label{eq:decom_prod}
  \prod_{\substack{s=1\\i_s\neq0}}^m\frac{1}{n_s^{p_s}}  =  \frac{1}{(n_{r(i)})^{p_{r(i)}}}\prod_{\substack{s=1\\i_s\neq0,s\ne r(i)}}^m\frac{1}{n_s^{p_s}},
\end{equation}
and relabeling the dummy index $t$ as $s$ in the last product
in~\eqref{eq:trap4}, we obtain
\begin{equation*}\label{eq:trap5}
E_{\mathrm{trap}}^{n}(F,\widetilde{g}) \leq C_4
	\sum_{\substack{i\in\Z_2^m\\ i\neq0}}(\log(F_{r(i)}))\frac{1}{(n_{r(i)})^{p_{r(i)}}}
\left(	\prod_{\substack{s=1\\i_s\neq 0,s\ne r(i)}}^m\frac{1}{n_s^{p_s}}\right) \left(	\prod_{\substack{s=1\\i_s\neq 0,s\ne r(i)}}^mF_s\right),
\end{equation*}        
or, combining the products with indexes $s\ne r(i)$
\begin{equation*}
E_{\mathrm{trap}}^{n}(F,\widetilde{g}) \leq C_4
	\sum_{\substack{i\in\Z_2^m\\ i\neq0}}(\log(F_{r(i)}))\frac{1}{(n_{r(i)})^{p_{r(i)}}}
        \left(	\prod_{\substack{s=1\\i_s\neq 0,s\ne r(i)}}^m\frac{F_s}{n_s^{p_s}}\right).
\end{equation*}        
Since, by hypothesis $\frac{F_s}{n_s}\leq a$, it follows
that the last product in this equation is less than or equal to
1. Thus, taking $p_{r(i)}$ to equal the order $p\geq 2$ in the
statement of the lemma, $p_{r(i)}=p \geq 2$, we obtain
\begin{align*}
E_{\mathrm{trap}}^{n}(F,\widetilde{g})&\leq C_4
	\sum_{\substack{i\in\Z_2^m\\ i\neq0}}\frac{\log(F_{r(i)})}{n_{r(i)}^{p}}.
\end{align*}
Since the sum on the right-hand side of this equation contains a finite number of terms, equation~\eqref{eq:trap-err} follows, and the proof is thus complete.
\end{proof}
 
The main theorem of this paper may now be established.
\begin{theorem}\label{theorem_main_hd}
  Let $P$ denote a positive integer and let $K_j$ in~\eqref{K-set}
  denote compact sets with a Lipschitz or piecewise smooth but
  possibly non-Lipschitz boundary. Further, let the (possibly
  discontinuous) function $f\in C_\mathrm{pw}^{\infty}(\K)$ be such
  that the restriction $f_j = f|_{K_j}$ of $f$ to the set $K_j$
  satisfies $f_j\in C^{\infty}(K_j)$, and let $\widetilde{f}_j$ be a
  smooth and $R$-periodic extension of $f_j$. Let
  $F\in\mathbb{N}^{m}$, $n\in\mathbb{N}^{m}$ such that
  \begin{equation}
    \label{n_and_F}
     F_r\leq a n_{r} \quad\mbox{for}\quad 1\leq r\leq m,\quad 0<a<1.
  \end{equation}
  Then, for any positive integer $p$ there is a constant $C$ such that
  the total TFF approximation error~\eqref{eq_tot_error_hd}  satisfies
  \begin{equation}\label{eq_tot_err_est} 
	 E_{\mathrm{tot}}^{n}\pFfT  
	\leq
	C\, \frac{\log(F_{\mathrm{max}})}{\left(F_{\mathrm{min}}\right)^{p}},
\end{equation}
where $F_{\mathrm{min}}$ and $F_{\mathrm{max}}$ are as defined
in~\eqref{tr_bnd}. If, additionally, the relations
$2\leq n_r/F_r \leq M$ hold for $1\leq r\leq m$ and for some constant
$M>0$, then
  \begin{equation}\label{eq_tot_err_est_N} 
	 E_{\mathrm{tot}}^{n}\pFfT  
	\leq
	C\, \frac{\log(n_{\mathrm{max}})}{\left(n_{\mathrm{min}}\right)^{p}}, \quad \mbox{where}\quad n_{\mathrm{max}} = \max_{1\leq r\leq m}\{n_r\}\quad \mbox{and}\quad n_{\mathrm{min}} = \min_{1\leq r\leq m}\{n_r\}.
\end{equation}

 \end{theorem}
 \begin{proof} 
   Using \eqref{int_pw} and \eqref{TFF_pw}, it suffices to establish
   \eqref{eq_tot_err_est} for $f \in C^{\infty}(K_j)$ for each $j$, which results directly from~\eqref{eq_tot_error_hd}, together
   with~\eqref{tr_bnd} and~\eqref{eq:trap-err}.
\end{proof} 

\section{Numerical Results}\label{Numerics}
This section presents a variety of numerical results that illustrate
the effectiveness and high-order accuracy of the TFF algorithms. All
computations were performed using a \CC\ implementation on a
single core of an Intel i7-11390H processor.

Throughout this section, we set
\begin{equation}
\label{eq:F}
F = n/2 \quad \text{(i.e., } F_r = n_r/2 \text{ for } 1 \leq r \leq m);
\end{equation}
(i.e., we use~\eqref{n_and_F} with $a=1/2$). We then report the errors
\begin{equation}\label{epsilon}
\varepsilon_n = \left|\A_{n,F}\ftilde - \A^{\text{ref}}[f]\right|
\quad \text{and} \quad
\varepsilon_{n,\infty} = \max\limits_{\x \in \D_h} \left| \B_{n,F}\fitilde(\x) - \B^{\text{ref}}[\phi](\x) \right| ,
\end{equation}
(with $\D_h$ as in equation~\eqref{eq:disc-pts}) arising from the numerical integral approximation~\eqref{TFF_pw}
and from the convolution approximation
\begin{equation}\label{TFF-conv-pw}
\B[\phi](\x) \approx \B_{n,F}\fitilde(\x),
\qquad \text{where} \qquad
\B_{n,F}\fitilde(\x) = T_{n}^{m,R}\!\left[G(\x - \cdot)\,\widetilde{\phi}(\cdot)\chi^K_F(\cdot)\right]
\end{equation}
to the integral and convolution
quantities~\eqref{Original_int_2D} and~\eqref{eq1}, respectively. Here $\A^{\text{ref}}[f]$ and
$\B^{\text{ref}}[\phi]$ denote either exact (closed-form) values of
the quantities $\A[f]$ and $\B[\phi]$, when available, or, otherwise,
highly accurate reference values thereof computed using the proposed
algorithm on a fine discretization. 
\begin{remark} Letting $\x$ in~\eqref{TFF-conv-pw} range over the Cartesian grid~\eqref{eq:disc-pts}
associated with the quadrature rule $T_{n}^{m,R}$, the approximation
$\B_{n,F}\fitilde(\x)$ assumes the form of a discrete convolution. This convolution can be evaluated efficiently at all grid points using the
Fast Fourier Transform (FFT)~\cite[Sec.~13.1]{press2007numerical}, at a
computational cost of $\bigO(N \log N)$, where $N$ is defined
in~\eqref{Original_int_2D_FS}. 
\end{remark}

 In all tabulated results, the
acronym ``NOC'' refers to the \textbf{n}umerical \textbf{o}rder of
\textbf{c}onvergence, computed as
\[
\log\left(\varepsilon_{n,\infty} / \varepsilon_{2n,\infty} \right) / \log(2).
\]
Numerical results for the integration and convolution problems are presented in Sections~\ref{tff-int} and~\ref{tff-conv}, respectively. For all results reported in those sections, the required Fourier coefficients of the characteristic function $\chi^\K$ were computed using the divergence-theorem-based algorithm described in Appendix~\ref{div-thm}. As predicted by Theorem~\ref{theorem_main_hd}, superalgebraic convergence is observed throughout. Section~\ref{fft-char}, in turn, presents numerical results illustrating the behavior of the $\bigO(F\log F)$ algorithm for the evaluation of Fourier coefficients. A detailed description and analysis of this algorithm will be provided in a forthcoming contribution.
	\begin{figure}[hbt!]
	\begin{center}
          \includegraphics[width=0.45\linewidth]{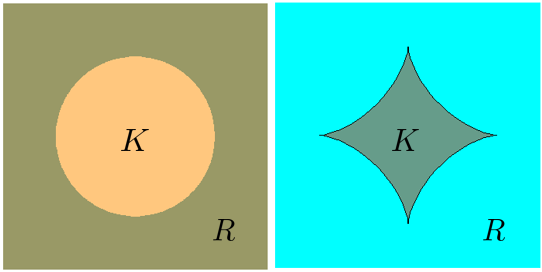}
          \caption{Circular and astroid integration domains $K$ and 
            periodicity rectangles (squares) $R$.\label{fig:Scatterers_volume}} \end{center}
\end{figure}

\subsection{TFF integration\label{tff-int}}

Results of a simple initial 1D test case are presented in
Table~\ref{1dtraptable}, which compares the convergence of the TFF
approach to that of the Trapezoidal Rule~\eqref{sec2:eq5} (TR) in the
computation of the integral of the function
\begin{equation}\label{f1}
  f_1(x) = \begin{cases}
    x^{3} + \sin(8x) + e^{-\cos^{2}(x)}&\mbox{if }x\in [-1,1]\\
    0 &\mbox{if } x\in  [-1.2, 1.2]\setminus [-1,1],
\end{cases}
\end{equation}
wherein $\K = [-1.2, 1.2]$, $R = [-1.5, 1.5]$ and
$\widetilde f:[-1.5, 1.5]\to \mathbb{R}$ is given via extension of $f$
by zero to the set $[-1.5, 1.5]\setminus [-1.2, 1.2]$. We note the
superalgebraic convergence order enjoyed by the TFF approach and the
first-order convergence rate exhibited by the trapezoidal-rule scheme
for the problem of integration of the piecewise smooth but
discontinuous function $f_1$.

\begin{table}[hbt!]
	\centering \scalebox{0.9}{
		\begin{tabular}{c| c| c|| c| c} \hline \hline
			\multirow{2}{*}{$n=N$} & \multicolumn{2}{c||}{TR}&
			\multicolumn{2}{c}{TFF}\\
			\cline{2-5}
			&$\varepsilon_{n}$ & NOC & $\varepsilon_{n}$ & NOC \\
			\hline
			4 & $1.6\cdot 10^{-1}$ & - & $9.6\cdot 10^{-2}$ & - \\
			8 & $7.1\cdot 10^{-2}$ & 1.20 & $8.4\cdot 10^{-3}$ & 3.52 \\
			16 & $3.8\cdot 10^{-2}$ & 0.89 & $2.5\cdot 10^{-3}$ & 1.73 \\
			32 & $1.8\cdot 10^{-2}$ & 1.05 & $3.1\cdot 10^{-4}$ & 3.02 \\
			64 & $9.4\cdot 10^{-3}$ & 0.98 & $1.4\cdot 10^{-4}$ & 1.15 \\
			128 & $4.7\cdot 10^{-3}$ & 1.01 & $1.9\cdot 10^{-7}$ & 9.50 \\
			256 & $2.3\cdot 10^{-3}$ & 0.99 & $8.5\cdot 10^{-14}$ & 21.1 \\
			512 & $1.2\cdot 10^{-3}$ & 1.00 & $1.3\cdot 10^{-15}$ & 5.99 \\
			1024 & $5.8\cdot 10^{-4}$ & 0.99 & $5.6\cdot 10^{-16}$ & 1.26 \\
			\hline \hline
	\end{tabular} }
	\caption{Convergence of the TR and TFF quadrature schemes for the
		evaluation of the function $f_1:\K \to\mathbb{R}$ defined
		in~\eqref{f1}.
	}
	\label{1dtraptable}
\end{table}
	\begin{table}[hbt!]
		\centering
		\scalebox{0.9}{
			\begin{tabular}{c| c| c| c| c| c| c| c|c}
				\hline 	\hline
				\multirow{3}{*}{$N=n_1\times n_2$} &
				\multicolumn{4}{c|}{Unit disc}&
				\multicolumn{4}{c}{Astroid-shaped domain}\\ \cline{2-9}
				& \multicolumn{2}{c|}{TR} & 
				\multicolumn{2}{c|}{TFF} 
				& \multicolumn{2}{c|}{TR} & 
				\multicolumn{2}{c}{TFF} \\
				\cline{2-9}
				& $\varepsilon_{n,\infty}$ & NOC  & $\varepsilon_{n,\infty}$ & NOC 
				& $\varepsilon_{n,\infty}$ & NOC  & $\varepsilon_{n,\infty}$ & NOC
				\\ \cline{2-3}
				\hline 	
				$4\times 4$ & $1.3\cdot 10^{-1}$ & - & $4.1\cdot 10^{-1}$ & -& $1.3\cdot 10^{0}$ & -& $4.4\cdot 10^{-1}$ & -\\
				$8\times 8$ & $7.1\cdot 10^{-2}$ & 0.90& $1.2\cdot 10^{-3}$ & 8.36 & $9.2\cdot 10^{-1}$ & 0.54& $6.3\cdot 10^{-3}$ & 6.12\\
				$16\times 16$ & $2.7\cdot 10^{-2}$ & 1.37& $6.4\cdot 10^{-5}$ & 4.27&  $9.5\cdot 10^{-3}$ & 6.60 & $1.8\cdot 10^{-5}$ & 8.43\\
				$32\times 32$ & $2.3\cdot 10^{-3}$ & 3.60& $5.5\cdot 10^{-6}$ & 3.54 &$6.0\cdot 10^{-2}$ & -2.66& $9.7\cdot 10^{-7}$ & 4.23\\
				$64\times 64$& $6.0\cdot 10^{-4}$ & 1.91 & $3.5\cdot 10^{-7}$ & 3.97& $1.6\cdot 10^{-2}$ & 1.90& $9.9\cdot 10^{-9}$ & 6.61\\
				$128\times 128$ & $1.3\cdot 10^{-3}$ & -1.11& $2.7\cdot 10^{-11}$ & 13.6 & $2.2\cdot 10^{-2}$ & -0.46& $9.3\cdot 10^{-12}$ & 10.1\\
				$256\times 256$ & $2.0\cdot 10^{-4}$ & 2.71& $3.2\cdot 10^{-14}$ & 9.72& $7.1\cdot 10^{-3}$ & 1.65& $2.1\cdot 10^{-15}$ & 12.1 \\
				$512\times 512$ & $1.6\cdot 10^{-5}$ & 3.61& $6.7\cdot 10^{-16}$ & 5.59& $6.6\cdot 10^{-4}$ & 3.44 & $1.2\cdot 10^{-16}$ & 4.15\\
				\hline\hline		
\end{tabular}}
\caption{Convergence of the TFF method for the evaluation of the integral of $f_2$ (defined in \eqref{exp_int}) over the unit disc and the astroid, using the computational domain $R=[-3,3]\times[-3,3]$. Errors are computed by comparison with the exact integral values in the case of the disc, and with a fine-mesh result in the case of the astroid.}
\label{T_f_disc-astroid}
\end{table}
	
\begin{table}[hbt!]
	\centering
	\scalebox{0.9}{
		\begin{tabular}{c| c| c| c| c| c| c| c|c}
			\hline 	\hline
			\multirow{3}{*}{$N=n\times n$} &
			\multicolumn{4}{c|}{Unit disc}&
			\multicolumn{4}{c}{Astroid-shaped domain}\\ \cline{2-9}
			& \multicolumn{2}{c|}{TR} & 
			\multicolumn{2}{c|}{TFF} 
			& \multicolumn{2}{c|}{TR} & 
			\multicolumn{2}{c}{TFF} \\
			\cline{2-9}
			& $\varepsilon_{n,\infty}$ & NOC  & $\varepsilon_{n,\infty}$ & NOC 
			& $\varepsilon_{n,\infty}$ & NOC  & $\varepsilon_{n,\infty}$ & NOC
			\\ \cline{2-3}
			\hline
$64\times 64$ & $4.4\cdot 10^{-1}$ & - & $4.5\cdot 10^{-1}$ & - & - & - & - & - \\
$128\times 128$ & $3.2\cdot 10^{-2}$ & 3.76 & $5.9\cdot 10^{-2}$ & 2.94 &$2.0\cdot 10^{0}$ & - & $2.1\cdot 10^{0}$ & -\\
$256\times 256$ & $6.3\cdot 10^{-3}$ & 2.35 & $2.8\cdot 10^{-4}$ & 7.73& $4.9\cdot 10^{-2}$ & 5.38& $5.5\cdot 10^{-3}$ & 8.57\\
$512\times 512$ & $3.1\cdot 10^{-3}$ & 1.01 & $2.6\cdot 10^{-11}$ & 23.4& $3.8\cdot 10^{-2}$ & 0.37& $3.0\cdot 10^{-3}$ & 0.87\\
$1024\times 1024$ & $5.5\cdot 10^{-4}$ & 2.51 & $1.7\cdot 10^{-16}$ & 17.3& $2.3\cdot 10^{-2}$ & 0.71& $1.2\cdot 10^{-10}$ & 24.5\\
$2048\times 2048$ & $3.5\cdot 10^{-4}$ & 0.66 & - & -& $2.0\cdot 10^{-3}$ & 3.52& $1.4\cdot 10^{-15}$ & 16.4\\
	\hline\hline		
	\end{tabular}}
	\caption{Same as Table~\ref{T_f_disc-astroid} but for the  oscillatory function $f_3$ (defined in \eqref{exp_int}) with $\omega=80$.}
	\label{T_oscl_disc-astroid}
\end{table}
\begin{remark}\label{trap-1st-ord}
  As is well known, the trapezoidal rule exhibits quadratic
  convergence for smooth functions defined on a one-dimensional
  interval. However, for piecewise smooth functions, the convergence
  typically degrades unless the discretization is chosen so that the
  discontinuity points coincide with grid nodes---which explains the
  first-order convergence displayed in the column labeled ``TR'' in
  Table~\ref{1dtraptable}. In the two-dimensional applications
  considered in this paper, further, where 2D Cartesian meshes
  generally do not align with discontinuity curves, the convergence
  behavior of the trapezoidal rule is often erratic, as illustrated in
  Table~\ref{T_f_disc-astroid}.
\end{remark}
We next consider integration problems of the
form~\eqref{Original_int_2D} over 2D domains $\K$ of varying
complexity. Specifically, we present illustrative examples involving
the integration of the functions
\begin{equation}\label{exp_int}
	f_2(\x) = \begin{cases}
		e^{-|\x|^{2}} &\text{if } \x \in \K,\\
		0 &\text{otherwise,}
	\end{cases}
        \quad\mbox{and}\quad
        f_3(\x) = \begin{cases}
          e^{-\ci \omega|\x|^{2}} &\text{if } \x \in \K,\\
          0 &\text{otherwise,}
	\end{cases}
        \quad (\omega\in\mathbb{R})
\end{equation}
over two different domains $\K$, namely, a unit disc and 
astroid-shaped domain, that are depicted in 
Figure~\ref{fig:Scatterers_volume}, and whose boundaries 
are given, respectively, by the following parametrizations 
defined in
the interval $0\leq t \leq 2\pi$:
\begin{align}
  &\x_1(t) = (\cos t, \sin t );\label{params1}\\
  &\x_2(t) = (\cos^{3}{t}, \sin^{3}{t}).\label{params2}
\end{align}
Table~\ref{T_f_disc-astroid} illustrates the
convergence properties exhibited by the TFF method for the integration
of the function $f_2$ over these two domains. For the disc, the
reference solution is given by a closed-form expression, while for the
astroid geometry it is obtained using a fine-mesh numerical
solution. The results confirm the high-order convergence of the
proposed algorithm and its effectiveness on complex domains. In
particular, they demonstrate that machine precision can be attained
for smooth functions with only moderate grid resolutions, and that
engineering accuracies result from coarse discretizations. The
high-order convergence is maintained even for oscillatory integrands
such as the function $f_3$ defined in \eqref{exp_int}, as illustrated
in Table~\ref{T_oscl_disc-astroid}.
\begin{table}[hbt!]
	\centering
	\scalebox{0.70}{
		\begin{tabular}{c| c| c| c| c| c| c| c| c| c| c| c| c }
			\hline \hline
			\multirow{3}{*}{$N=n\times n$} & 
            \multicolumn{6}{c|}{Disc} &
			\multicolumn{6}{c}{Astroid} \\
            \cline{2-13}
			& \multicolumn{3}{c|}{TR} &
			\multicolumn{3}{c|}{TFF} &
            \multicolumn{3}{c|}{TR} &
			\multicolumn{3}{c}{TFF}\\ 
			\cline{2-13}
			& $\epsilon_{N,\infty}$ &NOC & $T$ (sec.)& $\epsilon_{N,\infty}$ &NOC & $T$ (sec.) & $\epsilon_{N,\infty}$ &NOC & $T$ (sec.)& $\epsilon_{N,\infty}$ &NOC & $T$ (sec.) \\
			\hline \hline
			$128\times 128$ & $2.5\cdot 10^{0}$ & - & $4.0\cdot 10^{-3}$& $2.5\cdot 10^{0}$ & -& $5.0\cdot 10^{-3}$& $2.5\cdot 10^{0}$ & - &$4.0\cdot 10^{-3}$ & $2.5\cdot 10^{0}$ & -&$4.0\cdot 10^{-3}$\\
			$256\times 256$ & $3.2\cdot 10^{-1}$ & 2.96& $2.1\cdot 10^{-2}$& $1.2\cdot 10^{-1}$ & 4.37&$2.4\cdot 10^{-2}$ & $3.9\cdot 10^{-1}$ & 2.66 &$2.4\cdot 10^{-2}$ & $1.4\cdot 10^{-1}$ & 4.14 &$2.4\cdot 10^{-2}$\\
			$512\times 512$ & $1.1\cdot 10^{-1}$ & 1.51&$9.8\cdot 10^{-2}$& $8.0\cdot 10^{-4}$ & 7.24&$1.0\cdot 10^{-1}$& $1.2\cdot 10^{-1}$ & 1.72 &$1.0\cdot 10^{-1}$ & $8.0\cdot 10^{-4}$ & 7.50 &$1.0\cdot 10^{-1}$\\
			$1024\times 1024$ & $4.3\cdot 10^{-2}$ & 1.38&$6.0\cdot 10^{-1}$ & $4.2\cdot 10^{-9}$ & 17.5&$6.2\cdot 10^{-1}$& $5.2\cdot 10^{-2}$ & 1.20 &$6.0\cdot 10^{-1}$& $5.0\cdot 10^{-5}$ & 3.98 &$6.2\cdot 10^{-1}$\\
			$2048\times 2048$ & $2.2\cdot 10^{-2}$ & 0.96& $3.1\cdot 10^{0}$ & $5.8\cdot 10^{-15}$ & 22.8&$3.2\cdot 10^{0}$& $2.58\cdot 10^{-02}$ & 1.02 & $3.2\cdot 10^{0}$ & $5.7\cdot 10^{-16}$ & 36.4&$3.2\cdot 10^{0}$\\
			\hline \hline
	\end{tabular} }
	\caption{Spectral convergence of the TFF-based algorithm for the approximation of the convolution integral~\eqref{eq1} on the unit disc and the astroid, with $G$ and $\phi$ defined in~\eqref{ker-G} and with $R=[-3,3]\times[-3,3]$. Spectral convergence is observed in both cases, demonstrating that the cusp singularities present on the boundary of the astroid-shaped domain do not affect the convergence  of the method.} 
	\label{K1_conv_disc-astroid}
\end{table} 
\subsection{TFF-based smooth-kernel convolution\label{tff-conv}}
This section demonstrates the $\bigO(N\log N)$ computational cost and
superalgebraic convergence of the proposed TFF approach for the
evaluation of convolution integrals~\eqref{eq1} for infinitely smooth
integral kernels $G$ and densities $\phi$.

The first four test cases considered in this section concern TFF
convolutions over the Disc and Astroid domains $K$
depicted in Figure~\ref{fig:Scatterers_volume}, whose boundaries are
defined by the parametrizations given in
equations~\eqref{params1}-\eqref{params2}. These examples employ the
smooth kernel and density
\begin{equation}
G(\x - \y) = \frac{1}{2 \pi \sigma^2} \exp\left(-\frac{|\x - \y|^2}{2 \sigma^2}\right) \label{ker-G}
\quad \text{and} \quad
\phi(\x) =
\begin{cases}
1 + e^{-|\x|^2}, & \text{if } \x \in K, \\
0, & \text{otherwise},
\end{cases}
\end{equation}
with $\sigma = 0.5$.

Numerical values of the convolution integrals and associated errors,
evaluated for $\x$ over the complete Cartesian
discretization~\eqref{eq:disc-pts} of the square
$R = [-3,3]\times [-3,3]$, are presented in Table~\ref{K1_conv_disc-astroid}.  In these tables, $T$ denotes the total computational
time required in each case, with exception of the precomputation time
required for the evaluation of Fourier coefficients of the
characteristic function of the respective domains. In particular, these results
confirm the $\bigO(N \log N)$ computational complexity as well as the
high-order convergence established in
Theorem~\ref{theorem_main_hd}.


\subsection{Fast computation of Fourier coefficients: numerical examples}\label{fft-char}
As noted in Sections~\ref{sec:Intro} and~\ref{TFF}, the overall TFF methodology requires accurate evaluation of the Fourier coefficients of the characteristic function $\chi_K$, with coefficient errors commensurate with the prescribed overall error tolerance. In addition to the divergence-theorem-based procedure described in Appendix~\ref{div-thm}, which has been used to obtain the coefficient values used throughout the integration and convolution experiments reported in this paper, we have developed a substantially more efficient FFT-based algorithm for the evaluation of the characteristic-function Fourier coefficients. The latter method, whose description and analysis fall outside the scope of the present theoretical contribution, will be presented in a subsequent publication devoted to implementation aspects and large-scale numerical performance of the TFF approach. The numerical results presented in this section are included solely to provide representative evidence of the effectiveness of the aforementioned FFT-based coefficient-generation procedure. 

\begin{table}[H]
	\centering
	\scalebox{0.9}{
		\begin{tabular}{c|c|c|c| c| c| c|c}
			\hline \hline
			\multirow{2}{*}{$n_1=n_2$} & 
		    \multicolumn{6}{c|}{Maximum error} & 
 		      \multirow{2}{*}{$T$ (sec.)} \\
			\cline{2-7}
			& $n_1/2$ & $n_1/4$ &$n_1/8$ & $n_1/16$ &$n_1/32$ & $n_1/64$ & \\
			\hline \hline
16 & $4.6\cdot10^{-3}$  & $1.8\cdot10^{-3}$ & $1.3\cdot10^{-3}$ & $1.3\cdot10^{-3}$ & $1.3\cdot10^{-3}$ & $1.3\cdot10^{-3}$& $6.6\cdot10^{-3}$ \\
32 & $1.4\cdot10^{-3}$  & $4.4\cdot10^{-4}$ & $4.4\cdot10^{-4}$ & $4.4\cdot10^{-4}$ & $4.2\cdot10^{-4}$ & $4.2\cdot10^{-4}$ & $2.0\cdot10^{-2}$\\
64 & $2.7\cdot10^{-4}$  & $6.8\cdot10^{-5}$ & $3.3\cdot10^{-5}$ & $3.3\cdot10^{-5}$ & $3.3\cdot10^{-5}$ & $3.3\cdot10^{-5}$& $7.4\cdot10^{-2}$ \\
128 & $6.0\cdot10^{-5}$ & $1.4\cdot10^{-7}$ & $1.1\cdot10^{-8}$ & $6.9\cdot10^{-9}$ & $6.2\cdot10^{-9}$ & $6.0\cdot10^{-9}$  & $2.6\cdot10^{-1}$\\
256 & $7.6\cdot10^{-5}$  & $3.2\cdot10^{-8}$ & $2.0\cdot10^{-11}$ & $1.3\cdot10^{-12}$ & $9.4\cdot10^{-13}$ & $9.4\cdot10^{-13}$ & $6.7\cdot10^{-1}$\\
512 & $2.9\cdot10^{-5}$ & $9.3\cdot10^{-7}$ & $7.4\cdot10^{-11}$ & $7.4\cdot10^{-13}$ & $2.3\cdot10^{-13}$ & $2.0\cdot10^{-13}$ & $1.9\cdot10^{0}$ \\
1024 & $1.4\cdot10^{-5}$  & $7.3\cdot10^{-6}$ & $4.5\cdot10^{-10}$ & $1.2\cdot10^{-12}$ & $1.2\cdot10^{-13}$ & $1.2\cdot10^{-13}$ & $5.5\cdot10^{0}$\\
2048 & $5.4\cdot10^{-6}$ & $3.2\cdot10^{-6}$ & $1.4\cdot10^{-8}$ & $2.9\cdot10^{-12}$ & $6.3\cdot10^{-14}$ & $2.6\cdot10^{-14}$ & $1.6\cdot10^{1}$ \\
4096 & $2.1\cdot10^{-6}$  & $1.5\cdot10^{-6}$ & $3.6\cdot10^{-7}$ & $1.9\cdot10^{-11}$ & $7.4\cdot10^{-14}$ & $1.5\cdot10^{-14}$& $4.7\cdot10^{1}$ \\
    \hline \hline
	\end{tabular} } 
	\caption{Convergence of the Fourier coefficients $c_\ell$, $\ell = (\ell_1,\ell_2)$, of the characteristic function of unit disc produced by the rapidly convergent $\bigO(N\log N)$  coefficient-generation procedure ($N= (n_1 +1)\times (n_2 +1)$) over the frequency ranges $-\frac{n_1}{2^j}\le \ell_1,\ell_2 \le \frac{n_1}{2^j}$ ($1\leq j\leq 6$).}
    \label{T:fft-disc}
\end{table}

Tables~\ref{T:fft-disc}--\ref{T:fft-astroid} illustrate the accuracy and computational performance of the rapidly convergent $\bigO(N\log N)$ coefficient-generation procedure. For a given discretization size $N=(n_1 +1)\times (n_2+1)$, the tables report the maximum error in the computed Fourier coefficients over several spectral ranges: the columns labeled $n_1/2$, $n_1/4$, \ldots, $n_1/64$ correspond to the maximum error among all Fourier coefficients whose indices satisfy 
\[ 
-\frac{n_1}{2^j}\le \ell_1,\ell_2 \le \frac{n_1}{2^j}, \quad 1\leq j\leq 6,
\]  
respectively. Thus, for each fixed discretization size $N$, the FFT is performed on the full $n_1\times n_2$ grid, while the reported errors are measured over progressively smaller subsets of Fourier modes. As expected, for a fixed discretization, lower-frequency coefficients are resolved more accurately than higher-frequency ones. The tables therefore illustrate how the FFT-based algorithm may be used in practice to compute Fourier coefficients over a prescribed spectral range with a desired accuracy. Given a target frequency range and an error tolerance $\varepsilon$, one simply selects $n_1=n_2$ sufficiently large so that the reported error for the corresponding spectral range falls below the tolerance $\varepsilon$. The accompanying timing data indicate that these gains in accuracy are achieved at a computational cost consistent with the asymptotic $\bigO(N\log N)$ complexity of the FFT-based algorithm.

Tables~\ref{T:fft-disc}-\ref{T:fft-astroid} present results for the unit-disc geometry (whose boundary is smooth), a square domain (which possesses corner singularities), and the astroid (whose boundary includes cusps). In all three cases, the coefficient errors exhibit rapid convergence under mesh refinement and can be driven to near machine precision with sufficiently fine discretizations, even for domains whose boundaries contain geometric singularities. The results therefore indicate that the coefficient-generation procedure remains robust across the full range of smooth and non-smooth geometries considered in this paper. 

\begin{table}[H]
	\centering
	\scalebox{0.9}{
		\begin{tabular}{c|c|c|c| c| c| c|c}
			\hline \hline
			\multirow{2}{*}{$n_1=n_2$} & 
		      \multicolumn{6}{c|}{Maximum error} & 
 		      \multirow{2}{*}{$T$ (sec.)} \\
			\cline{2-7}
			& $n_1/2$ & $n_1/4$ &$n_1/8$ & $n_1/16$ &$n_1/32$ & $n_1/64$ & \\
			\hline \hline
16 & $8.1\cdot10^{-3}$ & $1.1\cdot10^{-3}$ & $6.9\cdot10^{-4}$ & $3.5\cdot10^{-4}$ & $3.5\cdot10^{-4}$ & $3.5\cdot10^{-4}$ & $1.0\cdot10^{-2}$ \\
32 & $4.6\cdot10^{-3}$ & $1.8\cdot10^{-4}$ & $1.8\cdot10^{-4}$ & $1.8\cdot10^{-4}$ & $1.8\cdot10^{-4}$ & $1.8\cdot10^{-4}$ & $3.4\cdot10^{-2}$ \\
64 & $6.5\cdot10^{-3}$ & $2.2\cdot10^{-5}$ & $3.4\cdot10^{-6}$ & $3.4\cdot10^{-6}$ & $3.4\cdot10^{-6}$ & $3.4\cdot10^{-6}$ & $1.4\cdot10^{-1}$ \\
128 & $3.0\cdot10^{-2}$ & $6.5\cdot10^{-3}$ & $2.0\cdot10^{-5}$ & $3.3\cdot10^{-7}$ & $1.1\cdot10^{-7}$ & $1.1\cdot10^{-7}$ & $2.4\cdot10^{-1}$\\
256 & $1.4\cdot10^{-2}$  & $3.1\cdot10^{-3}$ & $9.7\cdot10^{-6}$ & $1.6\cdot10^{-7}$ & $1.6\cdot10^{-8}$ & $1.6\cdot10^{-8}$ & $5.7\cdot10^{-1}$\\
512 & $7.0\cdot10^{-3}$  & $1.5\cdot10^{-3}$ & $4.8\cdot10^{-6}$ & $3.9\cdot10^{-8}$ & $6.8\cdot10^{-9}$ & $6.8\cdot10^{-9}$ & $1.5\cdot10^{0}$\\
1024 & $3.4\cdot10^{-3}$  & $7.6\cdot10^{-4}$ & $2.4\cdot10^{-6}$ & $1.0\cdot10^{-8}$ & $5.2\cdot10^{-10}$ & $4.8\cdot10^{-10}$ & $4.2\cdot10^{0}$\\
2048 & $1.7\cdot10^{-3}$  & $3.8\cdot10^{-4}$ & $1.2\cdot10^{-6}$ & $3.1\cdot10^{-9}$ & $4.2\cdot10^{-10}$ & $4.2\cdot10^{-10}$ & $1.3\cdot10^{1}$\\
4096 & $8.5\cdot10^{-4}$ & $1.9\cdot10^{-4}$ & $5.9\cdot10^{-7}$ & $1.6\cdot10^{-9}$ & $7.1\cdot10^{-11}$ & $6.3\cdot10^{-11}$ & $4.7\cdot10^{1}$ \\
    \hline \hline
	\end{tabular} } 
	\caption{Same as Table~\ref{T:fft-disc} but for the characteristic function of the ``$45^\circ$-rotated'' square of side $\sqrt{2}$ (with corners at the points $(1,0)$, $(0,1)$, $(-1,0)$ and  $(0,-1)$).}
    \label{T:fft-square}
\end{table}

The computing times $T$ displayed in Tables~\ref{T:fft-disc}--\ref{T:fft-astroid}  are consistent with the $\mathcal{O}(N\log N)$ cost estimate while remaining noticeably smaller than the asymptotic $\mathcal{O}(N\log N)$ growth law over the range of discretizations shown. This behavior results from the fact that, for the problem sizes considered here, the total execution time is still dominated by auxiliary stages whose asymptotic growth is slower than $\bigO(N\log N)$ but whose proportionality constants are comparatively large. Consequently, the FFT-dominated asymptotic regime is not yet fully visible in the measured timings. Nevertheless, the results demonstrate that highly accurate Fourier coefficients can be generated at practical computational cost, even for domains containing corners and cusps. 
\begin{table}[H]
\centering
\scalebox{0.90}{
\begin{tabular}{c| c|c|c|c|c|c|c}
\hline \hline
\multirow{2}{*}{$n_1=n_2$}& 
\multicolumn{6}{c|}{Maximum error} & 
\multirow{2}{*}{$T$ (sec.)} \\
\cline{2-7}
& $n_1/2$& $n_1/4$& $n_1/8$&$n_1/16$&$n_1/32$&$n_1/64$&  \\
\hline \hline
16 & $7.5\cdot10^{-3}$& $6.0\cdot10^{-3}$& $6.0\cdot10^{-3}$& $6.0\cdot10^{-3}$& $6.0\cdot10^{-3}$ & $6.0\cdot10^{-3}$ & $1.1\cdot10^{-2}$\\
32 & $2.6\cdot10^{-3}$& $1.4\cdot10^{-3}$& $4.4\cdot10^{-4}$& $4.4\cdot10^{-4}$& $4.4\cdot10^{-4}$ & $4.4\cdot10^{-4}$ & $3.8\cdot10^{-2}$\\
64 & $1.4\cdot10^{-3}$& $4.2\cdot10^{-4}$& $6.9\cdot10^{-5}$& $2.7\cdot10^{-5}$& $2.7\cdot10^{-5}$& $2.7\cdot10^{-5}$ & $1.2\cdot10^{-1}$\\
128 & $5.7\cdot10^{-4}$& $1.4\cdot10^{-4}$& $2.7\cdot10^{-6}$& $1.9\cdot10^{-6}$& $1.7\cdot10^{-6}$& $1.7\cdot10^{-6}$ & $5.3\cdot10^{-1}$\\
256 & $1.0\cdot10^{-3}$& $3.7\cdot10^{-5}$ & $5.7\cdot10^{-7}$& $3.9\cdot10^{-7}$& $3.8\cdot10^{-7}$& $3.6\cdot10^{-7}$ & $2.1\cdot10^{0}$\\
512 & $1.9\cdot10^{-3}$ & $1.0\cdot10^{-3}$& $2.1\cdot10^{-6}$& $1.9\cdot10^{-7}$ & $1.8\cdot10^{-7}$ & $1.8\cdot10^{-7}$ & $9.0\cdot10^{0}$\\
1024 & $3.7\cdot10^{-4}$& $3.3\cdot10^{-4}$ & $1.1\cdot10^{-5}$ & $1.2\cdot10^{-8}$& $1.9\cdot10^{-9}$ & $1.9\cdot10^{-9}$ & $3.2\cdot10^{1}$\\
2048 & $1.3\cdot10^{-4}$& $1.1\cdot10^{-4}$& $3.1\cdot10^{-5}$ & $1.8\cdot10^{-8}$& $5.9\cdot10^{-10}$& $5.9\cdot10^{-10}$& $1.3\cdot10^{2}$\\
4096 & $4.3\cdot10^{-5}$& $3.4\cdot10^{-5}$& $3.4\cdot10^{-5}$ & $3.1\cdot10^{-8}$& $2.1\cdot10^{-10}$& $1.9\cdot10^{-10}$& $5.6\cdot10^{2}$\\
\hline \hline
\end{tabular} } 
\caption{Same as Table~\ref{T:fft-disc} but for the characteristic function of the astroid (Figure~\ref{fig:Scatterers_volume} and equation~\ref{params2}).}
\label{T:fft-astroid}
\end{table} 

\section*{Acknowledgments}
OPB and KYP gratefully acknowledge support from the Air Force Office
of Scientific Research (contracts FA9550-21-1-0373, FA9550-25-1-0015) and the National Science Foundation (contracts DMS-2109831 and DMS-2607322). A portion of this work was originally developed during Krishna Y. Poojara's Ph.D. studies at the Indian Institute of Science Education and Research (IISER) Bhopal. He acknowledges financial support from the Council of Scientific and Industrial Research, India (CSIR) through file no. 09/1020(0183)/2019-EMR-I.

\appendix

\section{Divergence-theorem-based evaluation of the Fourier coefficients of $\chi^K$\label{div-thm}}	
	This appendix describes the coefficient-generation procedure referenced in Sections~\ref{sec:Intro} and in the preface to Section~\ref{Numerics}. We present the algorithm in the 2D case; a corresponding algorithm for $m$-dimensional domains can be obtained  analogously.
    
   The proposed approach is based on an application of the divergence theorem. Specifically, the integral appearing in~\eqref{sec2:eq3},  may be expressed in the form
	\begin{equation}\label{sec2:eq6}
\int_{a_{2}}^{b_{2}}\int_{a_{1}}^{b_{1}} \chi^K\left(x_{1},
x_{2}\right) e^{-\frac{2\pi \ci}{b_{1}-a_{1}}\ell_1 x_{1}}
e^{-\frac{2\pi \ci}{b_{2}-a_{2}}\ell_2 x_{2}} d x_{1} dx_{2} =
\int_{K}e^{-\frac{2\pi \ci}{b_{1}-a_{1}} \ell_1 x_{1}} e^{-\frac{2\pi
\ci}{b_{2}-a_{2}}\ell_2 x_{2}} d x_{1} d x_{2}.\\
	\end{equation}
The right-hand integral in this equation can  be transformed into an integral over the boundary $\partial K$ of $K$ by a direct application of the divergence theorem~\cite[Section 5.8, Theorem 5.16]{evans2015measure}. We thus obtain
\begin{equation}\label{sec2:eq7} 
\int_{K}e^{-\frac{2\pi \ci}{b_{1}-a_{1}}\ell_1 x_{1}} e^{-\frac{2\pi \ci}
{b_{2}-a_{2}}\ell_2 x_{2}} d x_{1}d x_{2} = 
\begin{cases}\frac{-(b_{1}-a_{1})}{ 2\pi i
\ell_{1}} \int_{\partial K} e^{\frac{-2\pi \ci \ell_{1}
x_{1}}{b_{1}-a_{1}}} e^{\frac{-2\pi \ci \ell_{2} x_{2}}{b_{2}-a_{2}}}
dx_{2} & \text { if } \ell_1 \neq 0 \\ \int_{\partial K}
x_{1}e^{\frac{-2\pi \ci \ell_{2} x_{2}}{b_{2}-a_{2}}} d x_{2} & \text {
if } \ell_{1}=0.
\end{cases}\\
\end{equation}
The one dimensional integrals in (\ref{sec2:eq7}) can be
approximated with high-order accuracy by means of any rapidly convergent one
dimensional quadrature. In our implementations, we have used the high-order
Clenshaw-Curtis quadrature~\cite{sommariva2013fast}.
\begin{figure}[hbt!]
	\begin{center}
		\includegraphics[width=0.40\linewidth]{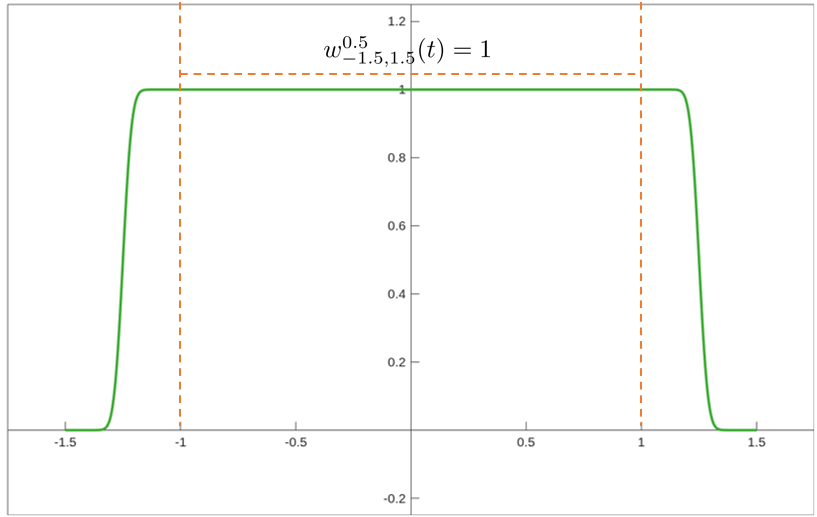}
		\includegraphics[width=0.40\linewidth]{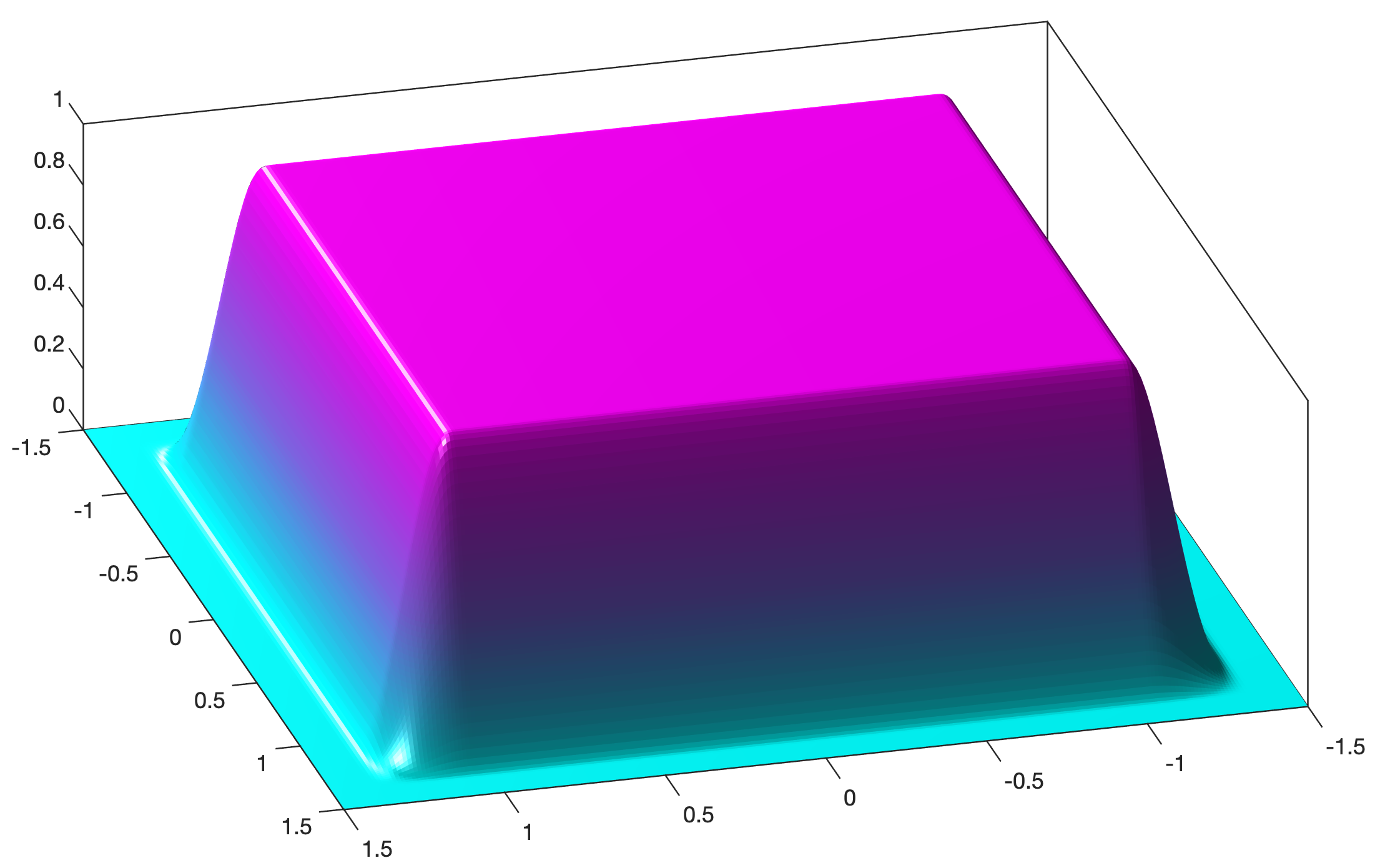}
	\end{center}
	\caption{Left: 1D smooth window function $w_{\alpha,\beta}^{s}$ with $[\alpha,\beta]=[-1.5,1.5]$ and $s=0.5$. Right: 2D smooth window function $\eta(\x)=w_{\alpha,\beta}^s(x_1)w_{\alpha,\beta}^s(x_2)$.\label{window}}
\end{figure}
\section{Window functions\label{wind_app}}
Per Remark~\ref{extension}, the numerical results presented in this paper were obtained using a smooth window function. Specifically we employed the  ERFC-based (complementary error function) window function given by
\begin{equation}\label{window_1d}
w_{\alpha,\beta}^s(t)= \begin{cases}
	1, & \text{ for } t\in[\alpha+s,\beta-s] \\ 
	\frac{1}{2}\mathrm{erfc}(-\rho+2(\rho/s) (t-\beta+s)), & \text { for } t\in(\beta-s,\beta) \\ 
	\frac{1}{2}\mathrm{erfc}(-\rho+2(\rho/s) (\beta+s-t)), & \text { for } t\in(\alpha,\alpha+s) \\ 
	0, & \text { for }t\in\R\setminus(\alpha,\beta),\end{cases}
\end{equation}
where $s$ denotes the window ``sharpness'',  where, for certain real numbers  $\alpha$ and $\beta$, $\alpha <\beta$, the window vanishes outside the interval $(\alpha,\beta)$, and where $s$ may take values in the interval $0<s<\beta-\alpha$. 

Strictly speaking this window function is discontinuous---but it may be thought of as infinitely differentiable ``up to a small tolerance''. In detail, the real number $\rho$ in equation~\eqref{window_1d} equals the inverse of the error function at a certain discontinuity tolerance value. In this paper we used the value $\rho = 5.805$, which yields a discontinuity tolerance equal to $w_{\alpha,\beta}^s(\alpha) = w_{\alpha,\beta}^s(\beta) = 1.1\times 10^{-16}$. (Here the quantity $1.1\times 10^{-16}$ equals $1/2$ of machine precision tolerance in the IEEE 754 Standard.)


\bibliographystyle{plain} \bibliography{TFF_approach.bib}
\end{document}